\documentclass{article}

\usepackage{etex}
\usepackage[utf8]{inputenc}
\usepackage[english]{babel}
\usepackage{footnote}
\makesavenoteenv{table}
\makesavenoteenv{tabular}
\usepackage{amsthm}
\usepackage{amsmath}
\usepackage{mathtools}
\usepackage{algorithm} 
\usepackage[noend]{algpseudocode}
\usepackage{amssymb}
\usepackage{extarrows}
\usepackage{graphicx}
\usepackage{array}
\usepackage{booktabs}
\usepackage{rotating}
\usepackage{multirow}
\usepackage{tikz}
\usepackage[font=footnotesize,skip=0pt]{caption}
\usepackage{pgfplots}
\usepgfplotslibrary{fillbetween}
\pgfplotsset{compat=newest}
\pgfplotsset{scaled x ticks=false} 
\pgfplotsset{
	small, 
	axis y line = left,
	axis x line = bottom,
	axis line style = {-latex, thick},
	xlabel style = {
		xshift = 14*\pgfkeysvalueof{/pgfplots/major tick length},
		yshift = 3*\pgfkeysvalueof{/pgfplots/major tick length},    
		anchor=north},  
	ylabel style = {
		yshift = -7*\pgfkeysvalueof{/pgfplots/major tick length},
		xshift = 15*\pgfkeysvalueof{/pgfplots/major tick length},
		rotate = -90,
		anchor=north},
} 

\makeatletter
\newcommand\footnoteref[1]{\protected@xdef\@thefnmark{\ref{#1}}\@footnotemark}
\makeatother

\newlength\figurewidth
\graphicspath{ {./figures/} }
\usepackage{float}
\usepackage{url}
\usepackage{color}
\usepackage[draft,nomargin]{fixme}
\fxsetup{theme=color,layout=inline} 
\usepackage{manfnt}                      % for the dangerous bend
\usepackage{hyperref}
\usepackage{enumerate} 
\usepackage[absolute]{textpos}
\usepackage{ragged2e}
\usepackage{tabularx}
\usepackage{here}
\usepackage{caption}
\usepackage{endnotes}

\allowdisplaybreaks

\newcommand{\RNum}[1]{\uppercase\expandafter{\romannumeral #1\relax}}

\newcommand{\A}{\mathbf{A}}

\newcommand{\NN}{\mathbb{N}}

\newcommand{\RR}{\mathbb{R}}

\newcommand{\EE}{\mathbb{E}}

\newcommand{\abs}[1]{|#1|}
\newcommand{\sign}[1]{\mathrm{sign}(#1)}

\DeclareMathOperator*{\argmin}{\arg\!\min}

\newtheorem{thm}{Theorem}[section]

\newtheorem{lem}[thm]{Lemma}
\newtheorem{rem}[thm]{Remark}
\newtheorem{as}{Assumption}

\newtheorem{ex}[thm]{Example}

\title{\textbf{\large Minimal error momentum Bregman-Kaczmarz}}
\author{Dirk Lorenz, Maximilian Winkler}
\date{November 2022}

\begin{document}
	
\author{Dirk A. Lorenz\thanks{Institute of Analysis and Algebra, TU Braunschweig, \href{mailto:d.lorenz@tu-braunschweig.de}{d.lorenz@tu-braunschweig}} \and Maximilian Winkler\thanks{Insitute of Analysis and Algebra, TU Braunschweig, \href{mailto:maximilian.winkler@tu-braunschweig.de}{maximilian.winkler@tu-braunschweig.de}}}
\maketitle
	
	\begin{abstract}
          The Bregman-Kaczmarz method is an iterative method which can solve strongly convex problems with linear constraints and uses only one or a selected number of rows of the system matrix in each iteration, thereby making it amenable for large-scale systems. To speed up convergence, we investigate acceleration by heavy ball momentum in the so-called dual update. Heavy ball acceleration of the Kaczmarz method with constant parameters has turned out to be difficult to analyze, in particular no accelerated convergence for the $\mathcal L_2$-error of the iterates has been proven to the best of our knowledge.
          Here we propose a way to adaptively choose the momentum parameter by a minimal-error principle similar to a recently proposed method for the standard randomized Kaczmarz method.
          The momentum parameter can be chosen to exactly minimize the error in the next iterate or to minimize a relaxed version of the minimal error principle. The former choice leads to a theoretically optimal step while the latter is cheaper to compute.
         We prove improved convergence results compared to the non-accelerated method. Numerical experiments show that the proposed methods can accelerate convergence in practice, also for matrices which arise from applications such as computational tomography.
	\end{abstract}
	
	\section{Introduction}
	
	We consider the problem
	\begin{align}
		\label{eqn:problem}
		\min \varphi(x) \qquad \text{s.t. } \A x = b, \qquad \A \in \RR^{m\times n}, \ b\in\RR^m
	\end{align}
	with a $\sigma$-strongly convex function $\varphi\colon\RR^n\to\RR$. To invoke stochastic optimization, we make use of the stochastic reformulation of the linear system given by 
	\begin{align}
		\label{eqn:stochastic_reformulation}
		\min f(x) := \EE[f_S(x)],
	\end{align} 
	where the expectation is taken over so-called \emph{sketching matrices} $S\sim\mathcal D$ which stem from a distribution $\mathcal D$ on the set of real-valued matrices with $m$ rows, and
	\begin{align*}
		f_S(x) := \frac12 \|S^T(\A x-b)\|_2^2. 
	\end{align*}
	The corresponding stochastic ``sketched Bregman method'' for problem~\eqref{eqn:problem} with stochastic representation~\eqref{eqn:stochastic_reformulation} is given by the update
	\begin{align}
		x_{k+1}^* &= x_k^* - \alpha_k \nabla f_{S_k}(x_k), \nonumber \\
		x_{k+1} &= \nabla \varphi^*(x_{k+1}^*). \label{eqn:SMD}
	\end{align}
	Here, $\alpha_k>0$ is a sequence of step sizes, $S_k$ is drawn from the distribution $\mathcal D$ and $\varphi^*\colon\RR^n\to\RR$ the convex conjugate function, which is everywhere finite and differentiable with $\sigma^{-1}$-Lipschitz continuous gradient. 	
	Method~\eqref{eqn:SMD} generalizes the \emph{block Bregman-Kaczmarz} method, initially named \emph{block sparse Kaczmarz}~\cite{LSW14, Petra15, LSTW22}, which iterates 
	\begin{align} 
		x_{k+1}^* &= x_k^* - \alpha_k \A_{(i_k)}^T(\A_{(i_k)}x_k - b_{(i_k)}), \nonumber \\
		x_{k+1} &= \nabla\varphi^*(x_{k+1}^*). \label{eqn:SparseKaczmarz}
	\end{align} 
	Here, the matrix $\A$ and the vector $b$ are partitioned as
	\begin{align*}
		\A = \begin{pmatrix}
			\A_{(1)} \\ \vdots \\ \A_{(c)}
		\end{pmatrix}, \qquad b =  \begin{pmatrix}
			b_{(1)} \\ \vdots \\ b_{(c)}
		\end{pmatrix}.
	\end{align*}
	Note that, if $\A_{(i_k)}$ consists of rows $r,r+1,...,s$, we recover the update by setting 
	\begin{align*}
		S_k = \begin{pmatrix}
			\hdots & e_{r} & e_{r+1} & \hdots & e_s & \hdots 
		\end{pmatrix}
	\end{align*}
	in~\eqref{eqn:SMD}, with the corresponding unit vectors.
	The block Bregman-Kaczmarz method~\eqref{eqn:SparseKaczmarz} uses only a selected number of rows of $\A$ and if $x_0^*=0$ and the step size is chosen as $\alpha_k = \sigma\|\A_{S_{i_k}}\|_F^{-2}$, it is known to converge linearly in expectation to the solution of problem~\eqref{eqn:problem}, see~\cite{LSTW22}. 
	For instance, with the $1$-strongly convex function $\varphi(x) = \lambda \|x\|_1 + \frac12 \|x\|_2^2$, the map $\nabla\varphi^*$ is the soft-shrinkage function 
	\begin{align*} 
		x \mapsto \mathcal S_\lambda(x) = \sign{x} \cdot \max(\abs{x}-\lambda, 0)
	\end{align*}
	and the method converges to the solution of the \emph{regularized basis pursuit problem}
	\begin{align*}
		\min \varphi(x) = \lambda \|x\|_1 + \frac12 \|x\|_2^2 \qquad \text{s.t. } \A x=b,
	\end{align*}
	which is sparse for $\lambda>0$ large enough, see~\cite{Friedlander08, LSW14}. 
	%Another example is the restricted negative entropy function 
	%\footnote{Here, we use the definition $0\log 0 = 0$.}
	%\begin{align*}
		%	\label{eqn:NegativeEntropyFunctionSimplex}
	%	\varphi(x) = \begin{cases}
		%	\sum_{i=1}^d x_i \log(x_i), & x \geq 0 \text{ and } \sum_{i=1}^n x_i = 1, \\
		%	+\infty, & \text{otherwise}
	%	\end{cases}
%	\end{align*}
%	which is $1$-strongly convex and finds solutions on the probability simplex, see~\cite[Example 9.17]{Beck17, GLW23}. 
		For the choice $\varphi(x) = \frac12 \|x\|_2^2$, we recover the \emph{block Kaczmarz method}
		\begin{align*}
		x_{k+1} = x_k - \alpha_k \A_{(i_k)}^T(\A_{(i_k)}x_k - b_{(i_k)}).
		\end{align*}
	
\subsection{Contribution and outline}
\label{sec:contrib-outline}

	As a first result we show in Section~\ref{sec:basic_method} that the iterates $x_k$ of~\eqref{eqn:SMD} converge linearly in expectation to the solution of~\eqref{eqn:problem}, if the $S_k$ are independent samples from an appropriate distribution $\mathcal D$ and $\alpha_k$ is a suitable sequence of step sizes.

	Next, in Section~\ref{sec:convergence} we enrich method~\eqref{eqn:SMD} by heavy ball momentum following the \emph{minimal error} strategy. That is, we consider a parameterized update 
	\begin{align*}
	x_{k+1}^*(\alpha, \beta) &= x_k^* - \alpha \nabla f_{S_k}(x_k) + \beta d_k^*, \nonumber \\
	x_{k+1}(\alpha, \beta) &= \nabla\varphi^*(x_{k+1}^*(\alpha, \beta)), 
\end{align*}
	and actually would like to find $\alpha_k=\alpha$ and $\beta_k=\beta$ such that the Bregman distance 
	\begin{align*}
		D_\varphi^{x_{k+1}^*(\alpha,\beta)}(x_{k+1}(\alpha,\beta),\hat x) = \varphi(\hat x) - \varphi(x_{k+1}(\alpha,\beta)) - \langle x_{k+1}^*(\alpha,\beta), \hat x - x_{k+1}(\alpha,\beta) \rangle
	\end{align*}
	to the solution $\hat x$ of~\eqref{eqn:problem} is minimized over $\alpha$ and $\beta$. 
	This idea is inspired by the \emph{exact-step sparse} Kaczmarz method~\cite{LSW14}, which for the special case $S_k=e_{i_k}$ and $\beta=0$ selects the step size $\alpha_k$ with the least Bregman distance to the solution $\hat x$ and has been demonstrated to significantly accelerate convergence in examples. However, we will argue that this problem can be computationally hard to solve. Therefore, in Section~\ref{subsec:min_err_momentum} we look at an easier problem, where we fix the step size $\alpha=\alpha_k$, which can be stochastic and dependent on expressions known at iteration $k$, and optimize just over $\beta$. 
	
	While our approach at first glance seems to require knowledge of the solution~$\hat x$, we can circumvent this problem by introducing an auxiliary variable $s_k\in\RR$, which is recursively updated together with $x_k,x_k^*$ and in each iteration equals $\langle \hat x, x_k^*-x_{k-1}^*\rangle$. It results that the momentum parameter $\beta_k$ is determined by the one-dimensional convex optimization problem
	\begin{align}
		\label{eqn:beta_problem}
		\beta_k \in \argmin_{\beta\in\RR}\varphi^*(x_k^* - \alpha_k \nabla f_{S_k}(x_k) +  \beta \big(x_k^*-x_{k-1}^*)\big) - \beta s_k
	\end{align}
	and the next $s_k$ can be found by setting $s_{k+1} =  \beta_k s_k - \alpha_k b_{i_k} $. We propose the corresponding heavy ball update as \emph{randomized Bregman-Kaczmarz with exact momentum} (BK-EM). 
	
	In addition, in Section~\ref{subsec:relaxed_min_err_momentum}, since problem~\eqref{eqn:beta_problem} is still costly to solve in general, we propose a relaxed approach by minimizing a quadratic upper bound of the Bregman distance $D_\varphi^{x_{k+1}^*(\alpha,\beta)}(x_{k+1}(\alpha,\beta),\hat x)$ jointly over $\alpha$ and $\beta$. This approach gives ``relaxed exact'' step sizes $\alpha_k$ and momentum parameters $\beta_k$ and we introduce this method as \emph{Bregman-Kaczmarz method with relaxed exact momentum} (BK-REM). 

        Finally, Section~\ref{sec:numerics} shows numerical experiments to illustrate the speedup in number of iterations and computational time.
	
\subsection{Related work}
In this article, we use the general stochastic reformulation of the linear system through sketching matrices as in~\cite{HSXZ23}. The framework of sketching matrices was first proposed by~\cite{GR15_rim}.

The standard error to be considered in convergence analysis for methods of randomized Kaczmarz - type and its accelerations is the $\mathcal L_2$-error $\EE[\|x_k-\hat x\|_2^2]$, where $\hat x$ is the solution of~\eqref{eqn:problem} and the expectation is taken over all possible combinations of sketching matrices. The seminal paper by Strohmer and Vershynin~\cite{SV09} proved the rate 
\begin{align*}
	\EE[\|x_k-\hat x\|_2^2] \leq \big(1-\frac{\sigma_{\min}(\A)^2}{\|\A\|_F^2}\big)^k \|x_0-\hat x\|_2^2.
\end{align*}
for the iterates $x_k$ of the Kaczmarz method, where $S_{k}=\{e_{i_k}\}$ and the $i$-th row $a_{i}$ is selected with probability $\|a_i\|_2^2/|\A\|_F^2$. For both the Kaczmarz and the Bregman-Kaczmarz method, is has been demonstrated that update averaging with appropriate weights gives a better convergence rate~\cite{MMNT21, LT22}. Accelerated rates have been also established for enhanced multistep versions~\cite{DHW23,NW13} and accelerations of the Kaczmarz method of Nesterov type~\cite{GHRS18,LW16, LRR19}, see also~\cite{IMN20} for an extension to systems of linear inequalities $\A x\leq b$. However, strict accelerated rates for the $\mathcal L_2$-error have not been shown for heavy ball momentum for the Kaczmarz method with a fixed momentum parameter. To the best of our knowledge, a convergence speedup has only been established for the quantity $\|\EE[x_k-\hat x]\|_2^2$, see \cite[Theorem 4]{LR20}, which is the error of the expected iterates and different from the above $\mathcal L_2$-error. Moreover, linear convergence of the $\mathcal L_2$ error has only been guaranteed for relatively restricted choices of the parameter, see~\cite{NJY22,LR20, Morshed20_3}. In~\cite{BCW22}, accelerated rates have been shown for heavy ball acceleration of minibatch Kaczmarz with sufficiently large batch sizes. Very recently, acceleration of Tseng type has been studied for the Bregman-Kaczmarz method via its interpretation as a dual coordinate descent method~\cite{LNT23}. Also recently, in~\cite{HSXZ23}, heavy ball acceleration with an adaptive momentum parameter has been derived for the Kaczmarz method. There, the authors propose to use the momentum parameter which moves the update closest to the exact solution, which is computable due to the linearity of the system. We call this idea the minimum error approach. It was shown that the method from~\cite{HSXZ23} converges at least with the Strohmer-Vershynin rate. In our work, we transfer this approach to the Bregman-Kaczmarz method and obtain \emph{exact momentum} (SRK-EM) and \emph{relaxed exact momentum} (SRK-REM) in the minimum error sense in Bregman distance.

	\section{Notation and Preliminaries}
	\label{sec:notation_and_preliminaries}
	
	\subsection{Notation}
	We set $\NN=\{0,1,...\}$.
	By $\mathcal R(B)$, $I_n$ and $e_i$, we denote the range of a matrix $B$, the identity matrix of dimension $n$ and the $i$-th unit vector, where its dimension is clear from the context.
	The linear space generated by $S\subset\RR^n$ is written as $\langle S\rangle$. The orthogonal projection onto a linear space $V\subset\RR^n$ is denoted by $P_V\colon\RR^n\to V$. If $V=\langle\{x\}\rangle$, we use the short-hand notation $P_V = P_x$. 
	
	\subsection{The stochastic reformulation}
	
	In this section, we make assumptions on the distribution $\mathcal D$ of the sketching matrices which ensure exactness of the stochastic reformulation~\eqref{eqn:stochastic_reformulation} of the linear system $\A x = b$ and give useful lemmas which deal with the sketching.
	
	\begin{as}
		\label{as:SST_finite_and_pd}
		The matrix $ \EE_{S\sim\mathcal D}[SS^T]$ has only finite entries and is positive definite.
	\end{as}

	Under assumption~\eqref{as:SST_finite_and_pd}, the reformulation is exact, as the following Lemma states. 
	
	\begin{lem}\cite[Lemma 2.2]{HSXZ23}
		\label{lem:exactness_stochastic_reformulation}
		Let Assumption~\ref{as:SST_finite_and_pd} hold. Then, the minimizers of~\eqref{eqn:stochastic_reformulation} are exactly the solutions of the linear system $\A x=b$. 
	\end{lem}
	
	A central object for the convergence analysis will be the matrix 

	\begin{align}
		\label{eqn:Def_M}
		M := \EE_{S\sim\mathcal D}\big[ \frac{SS^T}{\|\A^TS\|_2^2} \big].
	\end{align}

	\begin{as}
		\label{as:M_finite}
		The matrix $M$ in~\eqref{eqn:Def_M} has only finite entries.
	\end{as}
		
Note that Assumption~\ref{as:M_finite} requires in particular that $\A^TS$ is nonzero $\mathcal D$-almost surely. We discuss the definition of $M$ and Assumptions~\ref{as:SST_finite_and_pd} and~\ref{as:M_finite} in two examples.

	\begin{ex} 
		\mbox{} 
		\begin{enumerate}
			\item[(i)] \textbf{Single row sketching.} 
					If $S=e_i$ is sampled with probability $p_i$ such as in the randomized Kaczmarz method, we compute 
								\begin{align*}
						\EE[ SS^T ] = 
						\sum_{i=1}^m p_i e_ie_i^T =
						\begin{pmatrix}
							p_1 & & & \\
							& p_2 & & \\
							& & \ddots & \\
							& & & p_m
						\end{pmatrix}
					\end{align*}
					and
					\begin{align*}
						M = 
						\sum_{i=1}^m p_i \frac{e_ie_i^T}{\|a_i\|_2^2} =
						\begin{pmatrix}
							\frac{p_1}{\|a_1\|_2^2} & & & \\
							& \frac{p_2}{\|a_2\|_2^2} & & \\
							& & \ddots & \\
							& & & \frac{p_m}{\|a_m\|_2^2}
						\end{pmatrix}.
								\end{align*}
					 Hence, Assumption~\ref{as:SST_finite_and_pd} is equivalent to $p_i>0$ for $i=1,...,m$ and Assumption~\ref{as:M_finite} is fulfilled exactly if it holds $\|a_i\|_2> 0$ and $p_i>0$ for $i=1,...,m$. 
		
		\item[(ii)] \textbf{Block sketching.}
			If $S = \begin{pmatrix}
				\hdots & e_{r_i} & e_{r_i+1} & \hdots & e_{s_i} & \hdots 
			\end{pmatrix}$ is sampled with probability $p_i$ such as in the case of block Kaczmarz, Assumption~\ref{as:SST_finite_and_pd} is equivalent to $p_i>0$ for all~$i$ and Assumption~\ref{as:M_finite} is equivalent to the condition $\|A_{(i)}\|_F>0$ and $p_i>0$ for all~$i$. 
		\end{enumerate}
	\end{ex}

	The proof of the next lemma is along the lines of the one of~\cite[Lemma 2.5]{HSXZ23}.
	\begin{lem}
		Let Assumption~\ref{as:SST_finite_and_pd} hold. Then, the matrix $M$ is positive definite. 
	\end{lem}

	\begin{proof}
		Since the matrix $L:=\EE[SS^T]$ is positive definite, there exists $\epsilon>0$ such that for all $x\in\RR^n$ with $\|x\|_2=1$ it holds that $x^TLx>\epsilon$. For $k\in\NN$ we set 
		\[ \Lambda_{k,k+1} := \{ S: k \leq \|\A^TS\|_2^2 < k+1 \}. \]
		By Levi's theorem on monotone convergence, we have that
		\begin{align*}
			x^TLx = \int {x^TSS^Tx} \ \mathrm{d}\mathcal D(S) = \sum_{k=0}^\infty \int_{\Lambda_{k, k+1}} x^TSS^Tx \ \mathrm{d}\mathcal D(S).
		\end{align*}
		Since the left-hand side is finite and $x^TSS^Tx\geq 0$ for all $S$, there exists $k_0\in\NN$ such that 
		\begin{align*}
			\sum_{k=k_0}^\infty \int_{\Lambda_{k, k+1}} x^TSS^Tx \ \mathrm{d}\mathcal D(S) < \frac{\epsilon}{2},
		\end{align*}
		which means that 
		\begin{align*}
			\sum_{k=0}^{k_0-1}  \int_{\Lambda_{k, k+1}} x^TSS^Tx \ \mathrm{d}\mathcal D(S) = x^TLx - \sum_{k=k_0}^\infty \int_{\Lambda_{k, k+1}} x^TSS^Tx \ \mathrm{d}\mathcal D(S) > \frac{\epsilon}{2}.
		\end{align*}
	Hence, using Levi's theorem again we conclude that
		\begin{align*}
			x^TMx &= \sum_{k=0}^\infty \int_{\Lambda_{k, k+1}} \frac{x^TSS^Tx}{\|\A^TS\|_2^2} \ \mathrm{d}\mathcal D(S) \\
			&\geq  \sum_{k=0}^{k_0-1} \int_{\Lambda_{k, k+1}} \frac{x^TSS^Tx}{\|\A^TS\|_2^2} \ \mathrm{d}\mathcal D(S) \\
			&\geq \frac{1}{k_0} \sum_{k=0}^{k_0-1} \int_{\Lambda_{k, k+1}} x^TSS^Tx \ \mathrm{d}\mathcal D(S)\\
			&>  \frac{\epsilon}{2k_0}.
		\end{align*}
	\end{proof}

The following lemma will be also useful. 

\begin{lem}\cite[Lemma 2.3]{HSXZ23}
	\label{lem:kernel_with_tranpose}
	Assume that the linear system $Ax=b$ is consistent. Then for any real-valued matrix $S$ with $m$ rows and any vector $\tilde x\in\RR^n$ it holds that $\A^TSS^T(\A\tilde x-b)\neq 0$ if and only if $S^T(\A\tilde x-b)\neq 0$. 
\end{lem}

	\subsection{Convex analysis}
	
	We recall some concepts and properties of convex functions~\cite{BC17}.
	
	Let $\varphi\colon\RR^n\to\RR$ be convex. Since $\varphi$ is convex and finite everywhere, it is continuous and the set
 	\begin{align*}
 		\partial\varphi(x) = \{x^*\in\RR^n: \varphi(y) \geq \varphi(x) + \langle x^*,y-x\rangle \ \text{for all } y\in\RR^n\},
 	\end{align*}
 	is nonempty, convex and compact for every $x\in\RR^n$. 
 	
 	Throughout this article, we further assume that $\varphi$ is $\sigma$-strongly convex for some $\sigma>0$, which means that for all $x,y\in\RR^n$ and $x^*\in\partial \varphi(x)$ we have that
 	\begin{align*}
 		\varphi(y) \geq \varphi(x) + \langle x^*,y-x\rangle + \frac{\sigma}{2}\|y-x\|_2^2.
 	\end{align*}
 	The \emph{Bregman distance} between $x$ and $y$ with respect to $x^*\in\partial\varphi(x)$ is defined as
 	\begin{align}
 		\label{eqn:def_Bregman_distance}
 		D_\varphi^{x^*}(x,y) = \varphi(y)-\varphi(x)-\langle x^*,y-x\rangle.
 	\end{align}
	The \emph{convex conjugate} of $\varphi$ is defined as 
	\begin{align*}
		\varphi^*(x^*) = \sup_{x\in\RR^n} \langle x^*,x\rangle - \varphi(x), \qquad x^*\in\RR^n. 
	\end{align*}
	One can prove that the $\sigma$-strong convexity of $\varphi$ implies that the function $\varphi^*$ is finite everywhere and differentiable with $\sigma^{-1}$-Lipschitz continuous gradient. 
	Moreover, for all $x,x^*\in\RR^n$ it holds that $x^*\in\partial\varphi(x)$ if and only if $x=\nabla\varphi^*(x^*)$.
	In this case, the \emph{Fenchel equality} 
	\begin{align*}
		\varphi^*(x^*) = \langle x^*,x\rangle - \varphi(x)
	\end{align*}
	holds and the definition of the Bregman distance~\eqref{eqn:def_Bregman_distance} can be rewritten as
	\begin{align}
		\label{eqn:Bregman_dist_dual_formulation}
		D_\varphi^{x^*}(x,y) = \varphi^*(x^*) - \langle x^*,y\rangle + \varphi(y).
	\end{align}
	
	Finally, we can lower bound the Bregman distance between $x,y\in\RR^n$ for all $x^*\in\partial\varphi(x)$ by
	\begin{align}
		\label{eqn:Bregman_distance_lower_bound}
		D_\varphi^{x^*}(x,y) \geq \frac{\sigma}{2} \|x-y\|_2^2.
	\end{align}

	\section{The sketched Bregman-Kaczmarz method}
	\label{sec:basic_method}
	
	In this section, we study convergence of the sketched Bregman-Kaczmarz method for solving problem~\eqref{eqn:problem} given by
	\begin{align*}
		\min \varphi(x) \qquad \text{s.t. } \A x = b, \qquad \A \in \RR^{m\times n}, \ b\in\RR^m.
	\end{align*}
	We suppose that the system $\A x=b$ is consistent, i.e. has a solution, and $b\neq 0$. By strong convexity of $\varphi$, there is exactly one solution $\hat x$ to problem~\eqref{eqn:problem}.  For a distribution $\mathcal D$ on the set of real-valued matrices with $m$ rows, we consider the stochastic reformulation~\eqref{eqn:stochastic_reformulation} given by
	\begin{align*}
		\min f(x) := \EE[f_S(x)], \qquad f_S(x) = \frac12 \|S^T(\A x-b)\|_2^2.
	\end{align*} 
	The stochastic reformulation is exact under Assumption~\ref{as:SST_finite_and_pd} due to Lemma~\ref{lem:exactness_stochastic_reformulation}. 
	The gradient of $f_S$ can be easily calculated as 
	\begin{align*}
		\nabla f_{S}(x) = \A^T SS^T(\A x-b)
	\end{align*}
	and the sketched Bregman-Kaczmarz method is given by Algorithm~\ref{alg:SMD}. 
	\begin{algorithm}[H]
	\begin{algorithmic}[1]
		\State \textbf{Input:} $x_0^*=0\in\RR^n$, step sizes $\alpha_k>0$ and a distribution $\mathcal D$ on the set of matrices with $m$ rows
		\State \textbf{Initialization:} $x_0=\nabla\varphi^*(x_0^*)$
		\For{$k=0,1,...$}
		\State sample a random sketching matrix $S_k\sim {\mathcal D}$ (independent of $S_0,...,S_{k-1}$)
		\State $x_{k+1}^* = x_k^* - \alpha_k \nabla f_{S_k}(x_k) = x_k^* - \alpha_k \A^T S_kS_k^T(\A x_k-b) $        
		\State $x_{k+1} = \nabla\varphi^*(x_{k+1}^*)$          
		\EndFor
	\end{algorithmic}
	\caption{Basic method: Randomized Bregman-Kaczmarz method (BK) with sketching matrices}
	\label{alg:SMD}
\end{algorithm}	

Note that, if $x_0^*$ is chosen from $\mathcal R(\A^T)$, by induction it holds $x_k^*\in\mathcal R(\A^T)$ at every iteration $k$. Also, by subdifferential inversion we have $x_k^*\in\partial\varphi(x_k)$. As in~\cite{LSW14} we refer to $x_k$, $x_k^*$ as the primal variable and the dual variable, respectively.
	
	We now prove convergence of the $x_k$ in Algorithm~\ref{alg:SMD} to the solution $\hat x$ for two different step size sequences.

	To this end, we need to assume an error bound which relates the least-squares residual $\|\A x-b\|_2^2$ to the Bregman distance to the exact solution $\hat x$ for iterates of Algorithm~\ref{alg:SMD}. 
		
	\begin{as}
		\label{as:error_bound}	
		Let $\hat x$ be the solution of problem \eqref{eqn:problem}.
		There exists $\gamma>0$ such that for all $x\in\RR^n$ with $\partial\varphi(x)\cap\mathcal R(\A^T)\neq\emptyset$ and all $x^*\in \partial\varphi(x)\cap\mathcal R(\A^T)$ we have 
			\begin{align*}
				D_\varphi^{x^*}(x,\hat x) \leq \gamma \cdot \|\A x-b\|_2^2.
			\end{align*}
	\end{as}
	For the function $\varphi(x) = \lambda\|x\|_1 + \frac12 \|x\|_2^2$, it holds that $\mathrm{dom} \ \varphi = \RR^n$ and Assumption~\ref{as:error_bound} holds with an explicit constant $\gamma$ depending on $A$ and $\hat x$, see \cite[Lemma 3.1]{LS19}. See also \cite[Theorem 3.9]{LSTW22} for a sufficient condition for Assumption~\ref{as:error_bound} to hold.
	The next lemma will be used to prove convergence of Algorithm~\ref{alg:SMD} for suitable step sizes. 
	
	\begin{lem}
		\label{lem:rate_SMD_lemma}
		Let Assumptions~\ref{as:SST_finite_and_pd}-~\ref{as:M_finite} be fulfilled. Then, the iterates $x_k$, $x_k^*$ of Algorithm~\ref{alg:SMD} fulfill that 
		\begin{align*}
			D_\varphi^{x_{k+1}^*}(x_{k+1}, \hat x) \leq D_\varphi^{x_{k}^*}(x_{k}, \hat x) - \alpha_k \|S_k^T(\A x_k-b)\|_2^2 + \frac{\alpha_k^2}{2\sigma} \|\A^TS_kS_k^T(\A x_k-b)\|_2^2.
		\end{align*}
	\end{lem}

	\begin{proof}
		Since $\varphi$ is $\sigma$-strongly convex, $\varphi^*$ is $\sigma^{-1}$ smooth. Hence, using~\eqref{eqn:Bregman_dist_dual_formulation} and the descent lemma~\cite[Theorem 18.15(iii)]{BC17} give that
		\begin{align*}
			&D_\varphi^{x_{k+1}^*}(x_{k+1}, \hat x) \\ &=\varphi^*(x_{k+1}^*) - \langle x_{k+1}^*,\hat x\rangle + \varphi(\hat x) \\
			&\leq \varphi^*(x_k^*) + \langle \nabla\varphi^*(x_k^*), x_{k+1}^*-x_k^*\rangle + \frac{1}{2\sigma}\|x_{k+1}^* - x_k^*\|_2^2 - \langle x_{k+1}^*,\hat x\rangle +\varphi(\hat x).
		\end{align*}
	By inserting the update and using that $\nabla\varphi(x_k^*)=x_k$, we obtain
		\begin{align*}
			& D_\varphi^{x_{k+1}^*}(x_{k+1}, \hat x) \\ 
			& \leq \varphi^*(x_k^*) + \langle x_k, -\alpha_k \A^T S_kS_k^T(\A x_k-b) \rangle + \frac{\alpha_k^2}{2\sigma} \| \A^T S_kS_k^T(\A x_k-b) \|_2^2 \\
			& \hspace{0.5cm} - \langle x_k^*,\hat x\rangle + \alpha_k\langle \A^T S_kS_k^T(\A x_k-b), \hat x\rangle + \varphi(\hat x) \\
			&= D_\varphi^{x_{k}^*}(x_{k}, \hat x) - \alpha_k \langle x_k-\hat x, \A^T S_kS_k^T(\A x_k-b) \rangle + \frac{\alpha_k^2}{2\sigma} \| \A^T S_kS_k^T(\A x_k-b) \|_2^2 \\
			&= D_\varphi^{x_{k}^*}(x_{k}, \hat x) - \alpha_k\langle S_k^T\A(x_k-\hat x), S_k^T(\A x_k-b)\rangle + \frac{\alpha_k^2}{2\sigma} \| \A^T S_kS_k^T(\A x_k-b) \|_2^2
		\end{align*}
		and the assertion follows by using that $\A\hat x = b$. 
	\end{proof} 

	\begin{thm}
		\label{thm:rate_SMD}
		Let Assumptions~\ref{as:SST_finite_and_pd}-~\ref{as:error_bound} be fulfilled and let $x_k$, $x_k^*$ be the iterates of Algorithm~\ref{alg:SMD}. If at each iteration the step sizes are chose by
			\begin{align}
				\label{eqn:stepsize_SRK_non_min_err}
				\alpha_k = \begin{cases}
					\frac{\sigma}{\|\A^TS_k\|_2^2}, & \A^TS_k\neq 0, \\
					0, & \text{otherwise},
				\end{cases}
			\end{align}
			or
                        \begin{align}
				\label{eqn:stepsize_SRK_min_err}
				\alpha_k = \begin{cases}
					 \frac{\sigma\|S_k^T(\A x_k-b)\|_2^2}{\|\A^TS_kS_k^T(\A x_k-b)\|_2^2}, & S_k^T(\A x_k-b)\neq 0, \\
					0, & \text{otherwise},
				\end{cases}
			\end{align}
			is chosen, it holds that
			\begin{align}
				\label{eqn:SMD_rate_D}
				\EE\big[D_\varphi^{x_{k+1}^*}(x_{k+1}, \hat x)\big] \leq \Big( 1 - 	\frac{\sigma\lambda_{\min}(M)}{2\gamma} \Big)\EE\big[D_\varphi^{x_k^*}(x_k,\hat x)\big].
			\end{align} 
			In particular, the primal iterates $x_k$ converge to $\hat x$ in $\mathcal L_2$-sense with the linear rate
			\begin{align}
				\label{eqn:rate_SMD}
				\EE[\|x_k-\hat x\|_2^2] \leq \frac{2\varphi(\hat x)}{\sigma} \Big( 1 - \frac{\sigma\lambda_{\min}(M)}{2\gamma} \Big)^k.
			\end{align}
	\end{thm}
	\begin{proof}
		We insert the corresponding step sizes into Lemma~\ref{lem:rate_SMD_lemma}. 
		For step size~\eqref{eqn:stepsize_SRK_non_min_err}, again using Lemma~\ref{lem:rate_SMD_lemma}, this yields
		\begin{align*}
			D_\varphi^{x_{k+1}^*}(x_{k+1}, \hat x) 
			&\leq D_\varphi^{x_{k}^*}(x_{k}, \hat x) -  \frac{\sigma\|S_k^T(\A x_k-b)\|_2^2 }{\|\A ^TS_k\|_2^2} + \frac{\sigma\|\A^TS_kS_k^T(\A x_k-b)\|_2^2}{2\|\A^TS_k\|_2^4} \\
			&\leq D_\varphi^{x_{k}^*}(x_{k}, \hat x) - \frac{\sigma\|S_k^T(\A x_k-b)\|_2^2}{2\|\A^TS_k\|_2^2}.
		\end{align*}
		For step size~\eqref{eqn:stepsize_SRK_min_err}, we estimate
			\begin{align}
				D_\varphi^{x_{k+1}^*}(x_{k+1}, \hat x) &\leq D_\varphi^{x_{k}^*}(x_{k}, \hat x) - \frac{\sigma\|S_k^T(\A x_k-b)\|_2^4}{2\|\A^TS_kS_k^T(\A x_k-b)\|_2^2} \nonumber \\
				& \leq 	D_\varphi^{x_{k}^*}(x_{k}, \hat x) - \frac{\sigma\|S_k^T(\A x_k-b)\|_2^2}{2\|\A^TS_k\|_2^2}.
				\label{eqn:rate_SMD_convergence_without_expectation}
			\end{align}
			Since the $S_k$ are sampled independently, we can lower bound
			\begin{align*}
				\EE\Big[ \frac{\|S_k^T(\A x_k-b)\|_2^2}{\|\A^TS_k\|_2^2} \mid S_0,...S_{k-1} \Big] 
				&= 	\EE\Big[ \frac{\|S_k^T(\A x_k-b)\|_2^2}{\|\A^TS_k\|_2^2}  \mid S_0,...S_{k-1}  \Big] \\
				&= \bigl\langle \A x_k-b, \EE\Big[ \frac{S_kS_k^T}{\|\A^TS_k\|_2^2} \Big](\A x_k-b)  \bigr\rangle \\
				&= \langle \A x_k-b, M(\A x_k-b)\rangle \\
				&\geq \lambda_{\min}(M)\|\A x_k-b\|_2^2 \\
				&\geq \frac{\lambda_{\min}(M)}{\gamma} D_\varphi^{x_k^*}(x_k,\hat x),
			\end{align*}
			where the last step is due to Assumption~\ref{as:error_bound}. By the law of total expectation, we conclude
			\begin{align*}
				\EE\big[D_\varphi^{x_{k+1}^*}(x_{k+1}, \hat x)\big] \leq \Big( 1 - 		\frac{\sigma\lambda_{\min}(M)}{2\gamma} \Big)\EE\big[D_\varphi^{x_k^*}(x_k,\hat x)\big]
			\end{align*} 
				and hence inductively, 
				\begin{align*}
				 	\EE[D_\varphi^{x_k^*}(x_k,\hat x)] \leq \Big( 1 - \frac{\sigma\lambda_{\min}(M)}{2\gamma} \Big)^k\ D_\varphi^{x_0^*}(x_0,\hat x) = \varphi(\hat x) \Big( 1 - \frac{\sigma\lambda_{\min}(M)}{2\gamma} \Big)^k,
				 \end{align*} 
			 where the last equality is due to~\eqref{eqn:Bregman_dist_dual_formulation} and the fact that $x_0^*=0$. Finally by~\eqref{eqn:Bregman_distance_lower_bound}, strong convexity of $\varphi$ implies~\eqref{eqn:rate_SMD}.
	\end{proof}
		
	Note that step size~\eqref{eqn:stepsize_SRK_min_err} is always defined by Lemma~\ref{lem:kernel_with_tranpose}. 
	The first step size~\eqref{eqn:stepsize_SRK_non_min_err} generalizes the non-adaptive step size for block Kaczmarz~\cite{LSTW22}, the second one~\eqref{eqn:stepsize_SRK_min_err} the adaptive step size studied by~\cite{HSXZ23}. 

	\begin{ex}
		\label{ex:SRK_special_case_of_SMD}
		\mbox{}
		\begin{enumerate}
			\item[(i)] \textbf{Single row sketching.}  		
			In case that the sketching matrices are chosen as unit vectors $S_k=e_{i_k}$, Algorithm~\ref{alg:SMD} coincides with the Bregman-Kaczmarz method with single row sketching~\cite{LSW14}, and both step sizes~\eqref{eqn:stepsize_SRK_non_min_err} and~\eqref{eqn:stepsize_SRK_min_err} equal $\alpha_k = \frac{\sigma}{\|a_{i_k}\|_2^2}$. Indeed, in this case we have that 
			\begin{align}
				\nabla f_{S_k}(x_k) &= \A^TS_kS_k^T(\A x_k-b) = (\langle a_{i_k}, x_k\rangle - b_{i_k}) a_{i_k}, \nonumber \\
				S_k^T\A x_k &= \langle a_{i_k},x_k\rangle, \nonumber \\
				S_k^Tb &= b_{i_k}. 		
				\label{eqn:SRK_identities_with_Sk}
			\end{align}
			If the row $a_i$ is sampled with the probability $\|a_i\|_2^2/\|A\|_F^2$ suggested by Strohmer and Vershynin, we have that $\lambda_{\min}(M) = \frac{1}{\|A\|_F^2}$ and hence, Theorem~\ref{thm:rate_SMD} recovers the convergence rate from~\cite{LRR19} for $\varphi(x) = \lambda\|x\|_1 + \frac{1}{2}\|x\|_2^2$.
			\item[(ii)] \textbf{Block sketching.}   		
			Similarly, if $S_k = \begin{pmatrix}
				\hdots & e_{r_{i_k}} & e_{r_{i_k}+1} & \hdots & e_{s_{i_k}} & \hdots 
			\end{pmatrix}$, we calculate that
					\begin{align}
			\nabla f_{S_k}(x_k) &= A_{(i_k)}^T(A_{(i_k)}x_k-b_{(i_k)}), \nonumber \\
			S_k^T\A x_k &= A_{(i_k)}x_k-b_{(i_k)}, \nonumber \\
			S_k^Tb &= b_{(i_k)}. 		
			\label{eqn:Block_SRK_identities_with_Sk}
		\end{align}
		If $S_k$ is sampled with probability $\|A_{(i)}\|_F^2/\|A\|_F^2$, we have that $\lambda_{\min}(M) = \frac{1}{\|A\|_F^2}$ and the contraction factor in the convergence rate~\eqref{eqn:rate_SMD} is 
		\begin{align*}
			1 - \frac{\sigma\lambda_{\min}(M)}{2\gamma} = 1 - \frac{\sigma}{2\gamma\|A\|_F^2},
		\end{align*}
		 which to the best of our knowledge is novel in the literature.
		\end{enumerate}

	\end{ex}

	\section{Adaptive Heavy Ball acceleration}
	\label{sec:convergence}
	
	 We consider the parameterized update 
	\begin{align}
		x_{k+1}^*(\alpha,\beta) &= x_k^* - \alpha \nabla f_{S_k}(x_k) + \beta (x_k^*-x_{k-1}^*), \nonumber \\
		x_{k+1}(\alpha,\beta) &= \nabla\varphi^*(x_{k+1}^*), \label{eqn:momentum_parameterized}
	\end{align}
	that is, we incorporate heavy ball acceleration into the dual update. 
	
	In this section, we want to derive expressions for $\alpha=\alpha_k$ and $\beta=\beta_k$ which ensure convergence of the update. For notational purposes, we set 
	\begin{align}
		\label{eqn:def_d_y}
		y_k^*=x_k^*-\alpha_k\nabla f_{S_k}(x_k), \qquad y_k=\nabla\varphi^*(y_k^*), \qquad d_k^*=x_k^*-x_{k-1}^*.
	\end{align} 

By~\eqref{eqn:Bregman_dist_dual_formulation}, the Bregman distance of the primal variable to the solution~$\hat x$ of problem~\eqref{eqn:problem} can be calculated as
		\begin{align}
		&D_\varphi^{x_{k+1}^*(\alpha,\beta)}( x_{k+1}(\alpha,\beta), \hat x ) \nonumber \\  
		&= \varphi^*(x_k^* - \alpha \nabla f_{S_k}(x_k) + \beta d_k^*) - \langle  x_k^*- \alpha \nabla f_{S_k}(x_k) + \beta d_k^*, \hat x\rangle + \varphi(\hat x) \nonumber \\
		&= \varphi^*(x_k^* - \alpha \nabla f_{S_k}(x_k) + \beta d_k^*) - \langle x_k^*,\hat x\rangle + \alpha \langle S_k^T(\A x_k-b), S_k^Tb\rangle \nonumber \\
		&\hspace{0.5cm}  + \beta \langle d_k^*, \hat x\rangle + \varphi(\hat x).
		\label{eqn:momentum_parameterized_exact_bound}
		\end{align} 

	Hence, minimizing~\eqref{eqn:momentum_parameterized_exact_bound} over $\alpha$ and $\beta$ would give us momentum with minimal errors measured in Bregman distance, i.e. our update would fulfill the property
	\begin{align*}
		(\alpha,\beta) \in \argmin_{\alpha,\beta\in\RR} D_\varphi^{x_{k+1}^*(\alpha,\beta)}\big( x_{k+1}(\alpha,\beta), \hat x\big).
	\end{align*}
	However, there are two problems occurring here. First, problem~\eqref{eqn:momentum_parameterized_exact_bound} depends on the exact solution $\hat x$, which is not known. 
	Second, $\varphi^*$ is an arbitrary convex function which Lipschitz continuous gradient and hence, we expect minimization over two variables to be computationally expensive. For instance, the function $\varphi(x) = \lambda\|x\|_1 + \frac12 \|x\|_2^2$ has the conjugate $\varphi^*(x) = \frac12 \|\mathcal S_\lambda(x)\|_2^2$, which does not admit a closed form solution for minimizing~\eqref{eqn:momentum_parameterized_exact_bound} over $\alpha$ and $\beta$. 
	We will resolve the first problem of the $\hat x$-dependence completely by introducing a new scalar variable $s_k$, which is updated simultaneously with $x_k$, $x_k^*$ and in each iteration fulfills that $s_k = \langle x_k^*-x_{k-1}^*,\hat x\rangle$. In Section~\ref{subsec:min_err_momentum}, as a first approach to address the second problem, we optimize only over $\beta$ for a fixed step size $\alpha_k$ chosen as to Algorithm~\ref{alg:SMD}. For the function $\varphi(x) = \lambda\|x\|_1 + \frac12 \|x\|_2^2$, this can be done with reasonable effort by a sorting-based algorithm, exploiting the monotonicity of the partial derivative $\tfrac{\partial\varphi^*}{\partial\beta}$~\cite{LSW14}. 
	
	Therefore, as a second approach, in Section~\ref{subsec:relaxed_min_err_momentum} we optimize over an upper bound of~\eqref{eqn:momentum_parameterized_exact_bound}, obtained by the descent lemma, which is quadratic in $\alpha$ and $\beta$ and hence allows for easy joint minimization in both variables.
		
	\subsection{Minimum error steps}
		\label{subsec:min_err_momentum}
	
	In this section, we assume that the step size sequence $\alpha_k$ is fixed,\footnote{Here, $\alpha_k$ can be stochastic and adaptive, that is, dependent on expressions which are known in iteration $k$. We will prove convergence of the method in the case that $\alpha_k$ is chosen according to~\eqref{eqn:stepsize_SRK_non_min_err} or~\eqref{eqn:stepsize_SRK_min_err}, but the method can be formulated for any sequence $\alpha_k$ which can be implemented.} and for a momentum parameter $\beta>0$ we set
	\begin{align*} 
	x_{k+1}^*(\beta) &= x_k^* - \alpha_k \nabla f_{S_{i_k}}(x_k) + \beta(x_k^*-x_{k-1}^*), \\
	x_{k+1}(\beta) &= \nabla\varphi^*(x_{k+1}^*(\beta)).
\end{align*} 
	Here, we want to pursue the idea of choosing $\beta$ with a minimal error, that is, such that the primal update $x_{k+1}$ is closest to the solution~$\hat x$ in Bregman distance. In formulas, we seek for $\beta=\beta_k$ with
	\begin{align}
		\label{eqn:SMD_EM_min_property}
		\beta_k \in \argmin_{\beta\in\RR} D_\varphi^{x_{k+1}^*(\beta)}\big( x_{k+1}(\beta), \hat x \big).
	\end{align}
	To this end, using~\eqref{eqn:Bregman_dist_dual_formulation}, we rewrite the right-hand side~\eqref{eqn:SMD_EM_min_property} as
	\begin{align*}
		D_\varphi^{x_{k+1}^*(\beta)}\big( x_{k+1}(\beta), \hat x \big) &= \varphi^*(x_{k+1}^*(\beta)) - \langle x_{k+1}^*(\beta),\hat x\rangle + \varphi(\hat x) \\
		&= \varphi^*(y_k^*+\beta d_k^*) - \langle y_k^*,\hat x\rangle - \beta \langle d_k^*,\hat x\rangle + \varphi(\hat x),
	\end{align*}
	where $y_k^*$ and $d_k^*$ are defined as in~\eqref{eqn:def_d_y}.
	Consequently, the parameter $\beta$ is given as the solution to the nonsmooth convex optimization problem 
	\begin{align}
		\label{eqn:SMD_EM_beta_problem_with_x_hat}
		\beta_k \in \argmin_{\beta\in\RR} \varphi^*(y_k^*+\beta d_k^*)- \beta \langle d_k^*,\hat x\rangle.
	\end{align}
	At this point, problem~\eqref{eqn:SMD_EM_beta_problem_with_x_hat} depends on the solution $\hat x$. However, for
	\begin{align}
		s_k := \langle d_k^*,\hat x\rangle,
		\label{eqn:Def_sk_exact_momentum}
	\end{align}
	we have $s_0=0$, if we initialize $x_{-1}^*=x_0^*$, and we have the recursion 
	\begin{align}
		s_{k+1} &= \langle x_{k+1}^*-x_k^*,\hat x\rangle \nonumber \\
		&= \langle -\alpha_k\A^TS_kS_k^T(\A x_k-b), \hat x\rangle + \langle \beta_k d_k^*,\hat x\rangle \nonumber\\
		 &= -\alpha_k\langle S_k^T(\A x_k-b), S_k^Tb\rangle + \beta_ks_k. \label{eqn:SMD_EM_recursion_sk}
	\end{align}
	That is, we can set up a method which updates $x_k^*, x_k$ and $s_k$ with the update~\eqref{eqn:SMD_EM_recursion_sk} and determine $\beta_k$ by 
	\begin{align}
		\label{eqn:SMD_EM_beta_problem_without_x_hat}
		\beta_k \in \argmin_{\beta\in\RR} \varphi^*(y_k^*+\beta d_k^*)- \beta s_k.
	\end{align}	
	 This is summarized in Algorithm~\ref{alg:SMD_EM}. 
	 
	     \begin{algorithm}[H]
	 	\begin{algorithmic}[1]
	 		\State \textbf{Input:} $x_0^*=0\in\RR^n$ 
	 		 and a distribution $\mathcal D$ on the set of matrices with $m$ rows
	 		\State \textbf{Initialization:} $x_{-1}^*=x_0^*$, $x_0=\nabla\varphi^*(x_0^*)$ and $s_0 = 0$
	 		\For{$k=0,1,...$}
	 		\State sample a random sketching matrix $S_k\sim {\mathcal D}$ (independent of $S_0,...,S_{k-1}$)
	 		\State choose a step size $\alpha_k>0$
	 		\State $y_k^* = x_k^* - \alpha_k\nabla f_{S_k}(x_k) = x_k^* - \alpha_k\A^TS_kS_k^T(\A x_k-b)$
	 		\State $d_k^* = x_k^* - x_{k-1}^*$
	 		\If{$d_k^*\neq 0$}
	 		\State $\beta_k \in \argmin_{\beta\in\RR}\varphi^*(y_k^* +  \beta d_k^*) - \beta s_k$
	 		\Else 
	 		\State $\beta_k=0$
	 		\EndIf
	 		\State $x_{k+1}^* = y_k^* + \beta_k d_k^*$
	 		\State $x_{k+1} = \nabla\varphi^*(x_{k+1}^*)$
	 		\State $s_{k+1} =  -\alpha_k\langle S_k^T(\A x_k-b), S_k^Tb\rangle + \beta_ks_k. $                   
	 		\EndFor
	 	\end{algorithmic}
	 	\caption{Randomized Bregman-Kaczmarz method with exact minimal error momentum (BK-EM)}
	 	\label{alg:SMD_EM}
	 \end{algorithm}	
 
 	\begin{rem}
	 	By the derivation of Algorithm~\ref{alg:SMD_EM}, in iteration $k$ it holds
	 	 \begin{align*}
 			D_\varphi^{x_{k+1}^*}(x_{k+1},\hat x) \leq D_\varphi^{x_{k+1}^*(0)}(x_{k+1}(0),\hat x),
 		\end{align*}
 		where $x_{k+1}^*(0)$, $x_{k+1}(0)$ are the updates which Algorithm~\ref{alg:SMD} would compute. This means that each step of Algorithm~\ref{alg:SMD_EM} makes at least as much progress as Algorithm~\ref{alg:SMD} towards the solution $\hat x$ measured in the Bregman distance induced by~$\varphi$. We can not conclude at this point that Algorithm~\ref{alg:SMD_EM} converges faster than Algorithm~\ref{alg:SMD}. Nonetheless, in Theorem~\ref{thm:rate_SMD_EM} we will give an upper bound for the Bregman distances towards the solution which is as least as good as the one for Algorithm~\ref{alg:SMD} from Theorem~\ref{thm:rate_SMD}.
 	\end{rem}

 	\begin{ex}[Algorithm~\ref{alg:SMD_EM} for single row sketching]
 		\label{ex:SRK_EM_special_case_of_SMD_EM}
 		We revisit the case of the Bregman-Kaczmarz method with single row sketching (Example~\ref{ex:SRK_special_case_of_SMD}), that is, with $S_k = e_{i_k}$ and $\alpha_k = \frac{\sigma}{\|a_{i_k}\|_2^2}$. 
 		Then in view of the equations from~\eqref{eqn:SRK_identities_with_Sk}, the steps of Algorithm~\ref{alg:SMD_EM} are given as follows: 
 		\begin{align*}
 			d_k^* &= x_k^* - x_{k-1}^* \\
 			y_k^* &= x_k^* - \frac{\sigma(\langle a_{i_k},x_k\rangle - b_{i_k})}{\|a_{i_k}\|_2^2} a_{i_k} \\ 
 			\beta_k^* & \ \begin{cases}
 				\in \argmin_{\beta\in\RR}\varphi^*(y_k^* +  \beta d_k^*) - \beta s_k, & d_k^*\neq 0, \\ 
 				=0, & \text{otherwise} 
 			\end{cases} \\
 			x_{k+1}^* &= y_k^* + \beta_k d_k^*, \\
 			x_{k+1} &= \nabla\varphi^*(x_{k+1}^*), \\
 			s_{k+1} &=  -\frac{\sigma(\langle a_{i_k},x_k\rangle - b_{i_k})}{\|a_{i_k}\|_2^2} b_{i_k} + \beta_k s_k.
 		\end{align*}
 	\end{ex}

	We now prove convergence of Algorithm~\ref{alg:SMD}. 

	\begin{thm}
		\label{thm:rate_SMD_EM}
				Let Assumptions~\ref{as:SST_finite_and_pd}-~\ref{as:error_bound} be fulfilled. Let $x_k$, $x_k^*$ be the iterates of Algorithm~\ref{alg:SMD_EM} and $d_k^*$, $y_k^*$ and $y_k$ defined according to~\eqref{eqn:def_d_y}. Then it holds that
			\begin{align}
				\label{eqn:SMD_EM_rate_D}
				\EE\big[D_\varphi^{x_{k+1}^*}(x_{k+1}, \hat x)\big] \leq \Big( 1 - \frac{\sigma\lambda_{\min}(M)}{2\gamma} \Big)\EE\big[D_\varphi^{x_k^*}(x_k,\hat x)\big] - \frac{\sigma}{2}\EE\big[ \frac{\langle y_k-\hat x, d_k^*\rangle^2}{\|d_k^*\|_2^2} \big].
			\end{align} 
				In particular, if at each iteration $\alpha_k$ is chosen according to either~\eqref{eqn:stepsize_SRK_non_min_err} or~\eqref{eqn:stepsize_SRK_min_err}, $x_k$ converges in expectation to $\hat x$ in $\mathcal L_2$-sense with the linear rate~\eqref{eqn:rate_SMD}, that is, 
				\begin{align*}
					%\label{eqn:rate_SMD_EM}
					\EE[\|x_k-\hat x\|_2^2] \leq \frac{2\varphi(\hat x)}{\sigma} \Big( 1 - \frac{\sigma\lambda_{\min}(M)}{2\gamma} \Big)^k.
				\end{align*}
	\end{thm}
	\begin{proof}
		For all $\beta\in\RR$, by~\eqref{eqn:Bregman_dist_dual_formulation} and the descent lemma~\cite[Theorem 18.15(iii)]{BC17}, we have 
		\begin{align*}
			D_\varphi^{x_{k+1}^*(\beta)}(x_{k+1}(\beta), \hat x) &= \varphi^*(y_k^*+\beta d_k^*) - \langle y_k^*+\beta d_k^*,\hat x\rangle + \varphi(\hat x) \\
			&\leq \varphi^*(y_k*) + \beta \langle y_k, d_k^*\rangle + \frac{\beta^2}{2\sigma}\|d_k^*\|_2^2 - \langle y_k^*+\beta d_k^*,,\hat x\rangle + \varphi(\hat x) \\
			&= D_\varphi^{y_k^*}(y_k,\hat x) + \beta \langle y_k,d_k^*\rangle + \frac{\beta^2}{2\sigma}\|d_k^*\|_2^2 
				- \beta\langle d_k^*,\hat x\rangle \\
			&= D_\varphi^{y_k^*}(y_k,\hat x) + \beta\langle y_k-\hat x, d_k^*\rangle + \frac{\beta^2}{2\sigma}\|d_k^*\|_2^2.
		\end{align*}
		Minimizing this upper bound over $\beta$ gives the minimizer
		\begin{align*}
			\hat\beta = \frac{\sigma \langle \hat x-y_k, d_k^*\rangle}{\|d_k^*\|_2^2}.
		\end{align*}
		Hence, using the minimizing property~\eqref{eqn:SMD_EM_min_property} of $\beta_k$ we conclude that
		\begin{align*}
			D_\varphi^{x_{k+1}^*}(x_{k+1},\hat x)&= D_\varphi^{x_{k+1}^*(\beta_k)}(x_{k+1}(\beta_k), \hat x) \\
			&\leq D_\varphi^{x_{k+1}^*(\hat\beta)}(x_{k+1}(\hat\beta), \hat x) \\
			&= D_\varphi^{y_k^*}(y_k,\hat x) - \frac{\sigma\langle y_k-\hat x,d_k^*\rangle^2}{2\|d_k^*\|_2^2}.
		\end{align*}
		Taking expectation gives that 
		\begin{align*}
			\EE\big[ D_\varphi^{x_{k+1}^*}(x_{k+1},\hat x) \big] \leq \EE\big[ D_\varphi^{y_k^*}(y_k,\hat x) \big] - \frac{\sigma}{2} \EE\big[\frac{ \langle y_k-\hat x,d_k^*\rangle^2}{2\|d_k^*\|_2^2} \big].
		\end{align*}
		Finally, we note that $y_k$, $y_k^*$ equal $x_{k+1}$, $x_{k+1}^*$ from Algorithm~\ref{alg:SMD}, given that both updates start from $x_k$, $x_k^*$, and hence the assertion follows by Theorem~\ref{thm:rate_SMD}. 
	\end{proof}

	We see that the convergence estimate~\eqref{eqn:SMD_EM_rate_D} for the primal iterates of Algorithm~\ref{alg:SMD_EM} is at least as good as the estimate~\eqref{eqn:SMD_rate_D} for the primal iterates of Algorithm~\ref{alg:SMD}. For interpretation of the rates in case of single row sketching and block sketching, we refer to Example~\ref{ex:SRK_special_case_of_SMD}.

	\subsection{Relaxed minimum error steps}
	    \label{subsec:relaxed_min_err_momentum}
	
	Algorithm~\ref{alg:SMD_EM} suffers from the fact that determining $\beta_k$ as the solution of the problem~\eqref{eqn:SMD_EM_beta_problem_without_x_hat} is costly. 
		Using the descent lemma~\cite[Theorem 18.15(iii)]{BC17} again, we can estimate~\eqref{eqn:momentum_parameterized_exact_bound} from above by
	\begin{align} 
		&D_\varphi^{x_{k+1}^*(\alpha,\beta)}( x_{k+1}(\alpha,\beta), \hat x ) \nonumber \\ 
		&\leq \varphi^*(x_k^*) + \langle x_k, -\alpha \nabla f_{S_k}(x_k) + \beta d_k^*\rangle + \frac{1}{2\sigma}\|-\alpha \nabla f_{S_k}(x_k)+\beta d_k^*\|_2^2 \nonumber \\
		&\hspace{0.5cm} - \langle x_k^*,\hat x\rangle + \alpha \langle \nabla f_{S_k}(x_k),\hat x\rangle - \beta \langle d_k^*,\hat x\rangle + \varphi(\hat x) \nonumber  \\
		&= D_\varphi^{x_k^*}(x_k, \hat x) + \langle x_k-\hat x, -\alpha\nabla f_{S_k}(x_k) + \beta d_k^*\rangle + \frac{1}{2\sigma}\|-\alpha \nabla f_{S_k}(x_k)+\beta d_k^*\|_2^2.
		\label{eqn:momentum_parameterized_quadratic_upper_bound} 
	\end{align}
	 We note that~\eqref{eqn:momentum_parameterized_quadratic_upper_bound} is quadratic in $\alpha$ and $\beta$ and therefore easy to minimize. 

	Taking derivative of~\eqref{eqn:momentum_parameterized_quadratic_upper_bound} with respect to $\alpha$ yields 
	\begin{align*}
		0 &= - \langle x_k-\hat x, \nabla f_{S_k}(x_k)\rangle + \frac{1}{\sigma} \langle -\alpha_k \nabla f_{S_k}(x_k) + \beta_k d_k^*, -\nabla f_{S_k}(x_k)\rangle \\
		&= -\langle x_k-\hat x, \nabla f_{S_k}(x_k)\rangle + \frac{1}{\sigma} \big( \alpha_k\|\nabla f_{S_k}(x_k)\|_2^2 - \beta_k \langle d_k^*, \nabla f_{S_k}(x_k) \rangle \big) 
	\end{align*}
	and by rearranging we obtain the equation 
	\begin{align}
		\label{eqn:SRK_REM_partial_alpha_equation}
		 \alpha_k \|\nabla f_{S_k}(x_k)\|_2^2-  \beta_k\langle\nabla f_{S_k}(x_k), d_k^*\rangle = \sigma  \|S_k^T(\A x_k-b)\|_2^2.
	\end{align}
	 Taking derivative of~\eqref{eqn:momentum_parameterized_quadratic_upper_bound} with respect to $\beta$ yields
	\begin{align*}
		0 &= \langle x_k-\hat x, d_k^*\rangle + \frac{1}{\sigma} \langle - \alpha_k\nabla f_{S_k}(x_k) + \beta_k d_k^*, d_k^*\rangle \\
		&=	\langle x_k-\hat x,d_k^*\rangle + \frac{1}{\sigma} \big( - \alpha_k \langle \nabla f_{S_k}(x_k), d_k^*\rangle + \beta_k \|d_k^*\|_2^2 \big),
	\end{align*}
	which gives the second equation
	\begin{align}
			\label{eqn:SRK_REM_partial_beta_equation}
			-\alpha_k \langle \nabla f_{S_k}(x_k), d_k^*\rangle + \beta_k \|d_k^*\|_2^2 = \sigma \langle \hat x - x_k, d_k^* \rangle.
	\end{align}

	With the same derivation as for Algorithm~\ref{alg:SMD_REM}, we can again replace $\langle \hat x, d_k^*\rangle$ by~$s_k$, where we update $s_k$ according to~\eqref{eqn:SMD_EM_recursion_sk}. 
	
	We can now solve the system of linear equations~\eqref{eqn:SRK_REM_partial_alpha_equation}-\eqref{eqn:SRK_REM_partial_beta_equation} for $\alpha$ and $\beta$. The system is invertible, if 
	\begin{align*}
 		\|\nabla f_{S_k}(x_k)\|_2^2 \|d_k^*\|_2^2 - \langle \nabla f_{S_k}(x_k), d_k^*\rangle ^2 > 0,
	\end{align*}
	which is equivalent to the condition that the search directions $\nabla f_{S_k}(x_k)$ and $d_k^*$ are linearly independent. In case of linear dependence, we can set $\beta_k=0$ and using~\eqref{eqn:SRK_REM_partial_alpha_equation}, we obtain back $\alpha_k$ from~\eqref{eqn:stepsize_SRK_min_err}. In case of linear independence, the optimal solution $(\alpha_k, \beta_k)$ is given by

	\begin{align*}
		\alpha_k &= \sigma \frac{ \|S_k^T(\A x_k-b)\|_2^2 \|d_k^*\|_2^2  + \langle \nabla f_{S_k}(x_k), d_k^*\rangle (s_k-\langle x_k,d_k^*\rangle) }{\|\nabla f_{S_k}(x_k)\|_2^2 \|d_k^*\|_2^2 - \langle \nabla f_{S_k}(x_k), d_k^*\rangle ^2} \\
		\beta_k &= \sigma \frac{ \|S_k^T(\A x_k-b)\|_2^2 \langle \nabla f_{S_k}(x_k), d_k^*\rangle  + \|\nabla f_{S_k}(x_k)\|_2^2 (s_k-\langle x_k,d_k^*\rangle) }{\|\nabla f_{S_k}(x_k)\|_2^2 \|d_k^*\|_2^2 - \langle \nabla f_{S_k}(x_k), d_k^*\rangle ^2}.
	\end{align*} 

	We collect the steps in Algorithm~\ref{alg:SMD_REM}. 
	
		     \begin{algorithm}[H]
		\begin{algorithmic}[1]
			\State \textbf{Input:} $x_0^*=0\in\RR^n$ and a distribution $\mathcal D$ on the set of matrices with $m$ rows
			\State \textbf{Initialization:} $x_{-1}^*=x_0^*$, $x_0=\nabla\varphi^*(x_0^*)$ and $s_0 = 0$
			\State sample a random sketching matrix $S_k\sim {\mathcal D}$ (independent of $S_0,...,S_{k-1}$)
			\For{$k=0,1,...$}
			\State $d_k^* = x_k^* - x_{k-1}^*$
				\vspace{0.25cm}
			\If{$\|\nabla f_{S_k}(x_k)\|_2^2 \|d_k^*\|_2^2 - \langle \nabla f_{S_k}(x_k), d_k^*\rangle ^2>0$}
				\vspace{0.25cm}
			\State $\alpha_k = \sigma \frac{ \|S_k^T(\A x_k-b)\|_2^2 \|d_k^*\|_2^2  + \langle \nabla f_{S_k}(x_k), d_k^*\rangle (s_k-\langle x_k,d_k^*\rangle) }{\|\nabla f_{S_k}(x_k)\|_2^2 \|d_k^*\|_2^2 - \langle \nabla f_{S_k}(x_k), d_k^*\rangle ^2}$
			\State $\beta_k = \sigma \frac{ \|S_k^T(\A x_k-b)\|_2^2 \langle \nabla f_{S_k}(x_k), d_k^*\rangle  + \|\nabla f_{S_k}(x_k)\|_2^2 (s_k-\langle x_k,d_k^*\rangle) }{\|\nabla f_{S_k}(x_k)\|_2^2 \|d_k^*\|_2^2 - \langle \nabla f_{S_k}(x_k), d_k^*\rangle ^2}$
			\vspace{0.25cm}
			\Else 
			\State $\alpha_k=\sigma \frac{\|S_k^T(\A x_k-b)\|_2^2}{\|\nabla f_{S_k}(x_k)\|_2^2}$ 
			\State $\beta_k = 0$
			\EndIf
			\State $x_{k+1}^* = x_k^* - \alpha_k \nabla f_{S_k}(x_k) + \beta_k d_k^*$
			\State $x_{k+1} = \nabla\varphi^*(x_{k+1}^*)$
			\State $s_{k+1} =  -\alpha_k\langle S_k^T(\A x_k-b), S_k^Tb\rangle + \beta_ks_k. $                   
			\EndFor
		\end{algorithmic}
		\caption{Randomized Bregman-Kaczmarz method with relaxed minimal error momentum (BK-REM)}
		\label{alg:SMD_REM}
	\end{algorithm}

 	\begin{ex}[Algorithm~\ref{alg:SMD_REM} in case of single row sketching]
	\label{ex:SRK_REM_special_case_of_SMD_REM}
		We revisit again the case of the Bregman-Kaczmarz method with single row sketching (Example~\ref{ex:SRK_special_case_of_SMD}), that is, with $S_k=e_{i_k}$ and $\alpha_k = \frac{\sigma}{\|a_{i_k}\|_2^2}$. 
		 In view of the equations from~\eqref{eqn:SRK_identities_with_Sk}, the steps of Algorithm~\ref{alg:SMD_REM} are given as follows:
		\begin{align*}
			d_k^* &= x_k^* - x_{k-1}^*, \\
			t_k &= \begin{cases}
				\sigma \frac{ (\langle a_{i_k},x_k\rangle-b_{i_k}) \|d_k^*\|_2^2  + \langle a_{i_k},d_k^*\rangle (s_k-\langle x_k,d_k^*\rangle)}{\|a_{i_k}\|_2^2 \|d_k^*\|_2^2 - \langle a_{i_k},d_k^*\rangle^2}, & \|a_{i_k}\|_2^2 \|d_k^*\|_2^2 - \langle a_{i_k}, d_k^*\rangle^2 > 0, \\
				\sigma \frac{ (\langle a_{i_k},x_k\rangle - b_{i_k})}{\|a_{i_k}\|_2^2}, & \text{otherwise},
			\end{cases} \\
			\beta_k &= \begin{cases}
				\sigma \frac{ (\langle a_{i_k},x_k\rangle-b_{i_k})  \langle a_{i_k},d_k^*\rangle + \|a_{i_k}\|_2^2 (s_k-\langle x_k,d_k^*\rangle)  }{\|a_{i_k}\|_2^2 \|d_k^*\|_2^2 - \langle a_{i_k},d_k^*\rangle^2}, & \|a_{i_k}\|_2^2 \|d_k^*\|_2^2 - \langle a_{i_k}, d_k^*\rangle^2 > 0, \\
				\sigma \frac{ \langle a_{i_k},x_k\rangle - b_{i_k}}{\|a_{i_k}\|_2^2}, & \text{otherwise},
			\end{cases} \\
			x_{k+1}^* &= x_k^* - t_ka_{i_k} + \beta_k d_k^*, \\
			x_{k+1} &= \nabla\varphi^*(x_{k+1}^*), \\
			s_{k+1} &= - b_{i_k}t_k + \beta_k s_k.
		\end{align*}
		Note that we have $\alpha_k (\langle a_{i_k},x_k\rangle - b_{i_k}) = t_k$ and hence $\alpha_k\nabla f_{S_k}(x_k) = t_k a_{i_k}.$
	\end{ex}

	We now investigate convergence of Algorithm~\ref{alg:SMD_REM}. The following lemma will be useful. 
	
	\begin{lem}
		\label{lem:Proj_and_iterated_proj}
		Let $x,y,z\in\RR^n$ such that $(y,z)$ is linearly independent. Then it holds that
		\begin{align*}
			\| (I-P_{\langle \{y,z\}\rangle})x\|_2^2 
			=  \| (I-P_{z}) \circ (I-P_{y})x \|_2^2
			- \frac{\langle (I-P_{y})x, \ P_{z} \rangle^2 }{\|(I-P_{z})y\|_2^2}.
		\end{align*}
	\end{lem}
	
	\begin{proof}
		For 
		%$u = (I-P_{\langle \{y,z\}\rangle})x$, $u' = (I-P_{\langle\{z\}\rangle}) \circ (I-P_{y})x$ 
		\[ u = (I-P_{\langle \{y,z\}\rangle})x, \qquad u' = (I-P_z) \circ (I-P_{y})x \]
		and any $v\in \langle \{y,z\}\rangle^\perp$ we note that 
		\begin{align*}
			\langle u,y\rangle = \langle u,z\rangle = \langle u',z\rangle = 0
		\end{align*}
		and
		\begin{align*}
			\langle u',v\rangle = \langle x,v\rangle = \langle u,v\rangle.
		\end{align*}
		Hence, we have that
		\begin{align*}
			P_{\langle \{y,z\} \rangle^\perp}u' = P_{\langle \{y,z\} \rangle^\perp}x = P_{\langle \{y,z\} \rangle^\perp}u = u,
		\end{align*}
		which shows that 
		\begin{align*}
			\|u\|_2^2 = \|P_{\langle \{y,z\} \rangle^\perp}u'\|_2^2 = \|u'\|_2^2 - \|P_{\langle \{y,z\} \rangle}u'\|_2^2.
		\end{align*}
		Choosing the orthonormal basis $\big(z/\|z\|_2, \ (I-P_z)y/\|(I-P_z)y)\|_2\big)$ of the space $\langle \{y,z\}\rangle$, we further conclude that
		\begin{align*}
			\|u\|_2^2 &= \|u'\|_2^2 - \frac{\langle u',z\rangle^2}{\|z\|_2^2} - \frac{\langle (I-P_{z})(I-P_{y})x, \ (I-P_{z})y \rangle^2}{\|(I-P_{z})y\|_2^2} \\
			&= \|u'\|_2^2 - \frac{\langle (I-P_y)x, \ (I-P_{z})y \rangle^2}{\|(I-P_{z})y\|_2^2}\\
			&= \|u'\|_2^2 - \frac{\langle (I-P_y)x, \ 
				P_{z}y\rangle^2 }{\|(I-P_{z})y\|_2^2}.
		\end{align*}
	\end{proof}

	\begin{thm}
		\label{thm:rate_SMD_REM}
		Let Assumptions~\ref{as:SST_finite_and_pd}-~\ref{as:error_bound} be fulfilled. Let $x_k$, $x_k^*$ be the iterates of Algorithm~\ref{alg:SMD_REM} and 
		\begin{align}
			\label{eqn:def_tilde_yk}
			\tilde y_k = x_k - \frac{\|S_k^T(\A x_k-b)\|_2^2}{\|\A^TS_kS_k^T(\A x_k-b)\|_2^2} \A^TS_kS_k^T(\A x_k-b).
		\end{align}
		 Then it holds that
		\begin{align}
			\EE\big[ D_\varphi^{x_{k+1}^*}(x_{k+1}, \hat x)\big] 
				&\leq \big( 1 - \frac{\sigma\lambda_{\min}(M)}{2\gamma} \big) \EE\big[ D_\varphi^{x_k^*}(x_k, \hat x) \big] 
			-  \frac{\sigma}{2} \EE\big[ \frac{\langle \tilde y_k,d_k^*\rangle^2}{ \|d_k^*\|_2^2 } \big]
			\nonumber \\
			& \hspace{0.5cm} 
			- \frac{\sigma}{2} \EE\big[  \frac{ \langle \tilde y_k, d_k^*\rangle^2 \langle \nabla f_{S_k}(x_k), d_k^*\rangle^2  }{ \|d_k^*\|_2^2 \big(\|\nabla f_{S_k}(x_k)\|_2^2 \|d_k^*\|_2^2 - \langle \nabla f_{S_k}(x_k), d_k^*\rangle^2\big)  } \big]. \label{eqn:SMD_REM_rate_D}
		\end{align}
					In particular, $x_k$ converges in expectation to $\hat x$ in $\mathcal L_2$-sense with the linear rate~\eqref{eqn:rate_SMD}, that is, 
	\begin{align*}
		%\label{eqn:rate_SMD_EM}
		\EE[\|x_k-\hat x\|_2^2] \leq \frac{2\varphi(\hat x)}{\sigma} \Big( 1 - \frac{\sigma\lambda_{\min}(M)}{2\gamma} \Big)^k.
	\end{align*}
	\end{thm} 
	\begin{proof}
		By~\eqref{eqn:momentum_parameterized_quadratic_upper_bound} and the minimizing property of $\alpha_k$ and $\beta_k$ we conclude that 	
		\begin{align}
		&D_\varphi^{x_{k+1}^*}(x_{k+1}, \hat x) \nonumber \\
		&\leq 	D_\varphi^{x_k^*}(x_k, \hat x) + \langle x_k-\hat x, -\alpha\nabla f_{S_k}(x_k) + \beta d_k^*\rangle + \frac{1}{2\sigma}\|-\alpha \nabla f_{S_k}(x_k)+\beta d_k^*\|_2^2 \nonumber \\
		&= D_\varphi^{x_k^*}(x_k, \hat x) + \frac{\sigma}{2} \big( \| x_k-\hat x + \frac{1}{\sigma}(-\alpha_k\nabla f_{S_k}(x_k) + \beta_k d_k^*)\|_2^2 - \|x_k-\hat x\|_2^2 \big) \nonumber \\
		&\leq D_\varphi^{x_k^*}(x_k, \hat x) +  \frac{\sigma}{2} \inf_{\alpha,\beta\in\RR} \big( \| x_k-\hat x + (-\alpha\nabla f_{S_k}(x_k) + \beta d_k^*)\|_2^2 - \|x_k-\hat x\|_2^2 \big) \nonumber \\
		&\leq D_\varphi^{x_k^*}(x_k, \hat x) + \frac{\sigma}{2}\big(\| (I-P_{\langle \{\nabla f_{S_k}(x_k), d_k^*\}\rangle})(x_k-\hat x) \|_2^2 - \|x_k-\hat x\|_2^2 \big). \label{eqn:proof_SMD_REM_rate_estimate}
		\end{align}
		
		We can therefore apply Lemma~\ref{lem:Proj_and_iterated_proj} with $x_k-\hat x$ instead of $x$, $\nabla f_{S_k}(x_k)$ instead of $y$ and $d_k^*$ instead of $z$. We calculate the expressions on the right-hand side in Lemma~\ref{lem:Proj_and_iterated_proj} as
		\begin{align*}
			(I-P_y)x &= x_k-\hat x - \frac{\langle x_k-\hat x, \nabla f_{S_k}(x_k)\rangle}{\|\nabla f_{S_k}(x_k)\|_2^2} \nabla f_{S_k}(x_k) \\
			&= x_k - \hat x - \frac{\|S_k^T(\A x_k-b)\|_2^2}{\|\nabla f_{S_k}(x_k)\|_2^2} \nabla f_{S_k}(x_k), \\
			\| (I-P_z) \circ (I-P_y)x\|_2^2 &= \|\tilde y_k-\hat x\|_2^2 - \frac{\langle \tilde 	y_k,d_k^*\rangle^2}{\|d_k^*\|_2^2} \\ 			
			&= \|x_k-\hat x\|_2^2 - \frac{\|S_k^T(\A x_k-b)\|_2^4}{\|\nabla f_{S_k}(x_k)\|_2^2} - \frac{\langle \tilde y_k, d_k^* \rangle^2}{\|d_k^*\|_2^2}
		\end{align*}	
	and as in the proof of Theorem~\ref{thm:rate_SMD} we can upper bound 
	\begin{align*}
		\|(I-P_z) \circ (I-P_y)x\|_2^2 \leq \|x_k-\hat x\|_2^2 - \frac{\|S_k^T(\A x_k-b)\|_2^2}{\|\A^TS_k\|_2^2} 
		- \frac{\langle \tilde y_k,d_k^*\rangle^2}{ \|d_k^*\|_2^2 }.
	\end{align*}
	Finally, we compute the right term on the right-hand side in Lemma~\ref{lem:Proj_and_iterated_proj} as
			\begin{align*}
		\frac{\langle (I-P_{y})x, \ P_{z}y \rangle^2 }{\|(I-P_{z})y\|_2^2} 
			&= \frac{ \langle \tilde y_k, P_{d_k^*}\nabla f_{S_k}(x_k) \rangle^2 }{ \| \nabla f_{S_k}(x_k) - P_{d_k^*}\nabla f_{S_k}(x_k) \|_2^2} \\
			&= \frac{ \langle \tilde y_k, \frac{\langle d_k^*, \nabla f_{S_k}(x_k)\rangle}{\|d_k^*\|_2^2} d_k^* \rangle^2 }
			{ \| \nabla f_{S_k}(x_k) - \frac{\langle\nabla f_{S_k}(x_k), d_k^*\rangle}{\|d_k^*\|_2^2} d_k^* \|_2^2 } \\
			&= \frac{ \langle \tilde y_k, d_k^*\rangle^2 \langle \nabla f_{S_k}(x_k), d_k^*\rangle^2  }{ \|d_k^*\|_2^2 \big(\|\nabla f_{S_k}(x_k)\|_2^2 \|d_k^*\|_2^2 - \langle \nabla f_{S_k}(x_k), d_k^*\rangle^2\big)  }.
		\end{align*}
	After inserting all expressions into~\eqref{eqn:proof_SMD_REM_rate_estimate}, the assertion follows as in the proof of Theorem~\ref{thm:rate_SMD}. 
	\end{proof}

	As for Algorithm~\ref{alg:SMD_EM}, we see that the convergence estimate~\eqref{eqn:SMD_REM_rate_D} for the primal iterates of Algorithm~\ref{alg:SMD_REM} is at least as good as the bound~\eqref{eqn:SMD_rate_D} for Algorithm~\ref{alg:SMD}. Note that we can not directly relate the estimate~\eqref{eqn:SMD_REM_rate_D} to the estimate~\eqref{eqn:SMD_rate_D} for Algorithm~\ref{alg:SMD_EM}, as $y_k$ and $\tilde y_k$ are different in general. For interpretation of the rates in case of single row sketching and block sketching, we refer to Example~\ref{ex:SRK_special_case_of_SMD}.

	\section{Numerical experiments}
	\label{sec:numerics}
	
	In this section, we study the computational behaviour of our two proposed accelerations for the case of the sparse Kaczmarz method, that is, we choose $\varphi(x) = \lambda\|x\|_1 + \frac{1}{2}\|x\|_2^2$ and single row sketching $S_k=e_{i_k}$. All experiments are conducted in MATLAB R2022b on a macbook with 1,2 GHz Quad-Core Intel Core i7 processor and 16 GB memory. The code to reproduce our results can be found at \href{https://github.com/MaxiWk/Minimal-error-momentum-Bregman-Kaczmarz}{\texttt{https://github.com/MaxiWk/Minimal-error-momentum-Bregman-Kaczmarz}}. 
	
	In experiment (I), we evaluate the performance of the methods on artificial Gaussian systems. The data was generated by sampling a matrix $\A\in\RR^{m\times n}$ with entries from the standard normal distribution and a vector $\hat x\in\RR^n$ with $10$ nonzero entries, also from the standard normal distribution, at random positions uniformly distributed over $\{1,...,n\}$. The vector $b$ was set to $\A\hat x$ to ensure consistency of the system. In all examples, we sampled the $i$-th row of the matrix~$\A$ with probability $p_i=\|a_i\|_2^2/\|\A\|_F^2$. In Figure~\ref{fig:randn_quantiles}, we report the decay of the relative residual $\|\A x_k-b\|_2/\|b\|_2$ of the vanilla sparse Kaczmarz method (SRK), the exact-step sparse Kaczmarz method (ESRK) and our two proposed accelerations, namely exact momentum with step size $\alpha_k$ from vanilla sparse Kaczmarz (SRK-EM), and relaxed exact momentum (SRK-REM). First, by comparing the scaling of the horizontal axis in both subfigures, we observe that all three modifications of SRK accelerate the method for $\lambda=5$ significantly, and much more than for $\lambda=0.1$. Next, we compare performance of the methods in the right subfigure. The method with relaxed momentum (SRK-REM) gives better acceleration than the (SRK-EM) method, and in the median it is comparable to the acceleration achieved by the ESRK method, which has more costly iterations. It is also visible that the ESRK residual decays with a large variance over the random instances, which is already known from~\cite{LS19}. This appears to be less the case for the SRK-REM method. From Table~\ref{tab:randn_runtime} we see that the SRK-REM method needs by far the least computation time to achieve a relative residual of $10^{-6}$. We also note that it did not pay off to introduce exact momentum with the sparse Kaczmarz step size (SRK-EM), compared to using the exact sparse Kaczmarz step size without momentum (ESRK).

	\begin{figure}[htb]
		
		\begin{center}
			
		\begin{tabular}{rl}
			
			\begin{tikzpicture}
				\begin{axis}[
					thin,smooth,no markers,
					ymode=log,     
					width = 0.45\textwidth, 
					xlabel = {$k$}, 
					ylabel = {$\|\A x_k-b\|_2/\|b\|_2$}, 
					xtick = {5e4}, 
					ytick = {1e-12, 1e-8, 1e-4, 1e0},  
					ymin = 1e-16, ymax = 1e3,
					]
					
					\addplot+[name path=SRKmin, draw=none] table[y=min_res_srk]{MATLAB/randn_sparse/lambda=.1/res_over_iter_quantiles.txt} ;
					\addplot+[name path=SRKmax, draw=none] table[y=max_res_srk]{MATLAB/randn_sparse/lambda=.1//res_over_iter_quantiles.txt} ;
					\addplot+[black, thick] table[y=median_res_srk]{MATLAB/randn_sparse/lambda=.1//res_over_iter_quantiles.txt} ;
					\addplot[black!10] fill between[of=SRKmin and SRKmax];
					\addplot+[name path=SRK25, draw=none] table[y=quant25_res_srk]{MATLAB/randn_sparse/lambda=.1//res_over_iter_quantiles.txt} ;
					\addplot+[name path=SRK75, draw=none] 	table[y=quant75_res_srk]{MATLAB/randn_sparse/lambda=.1//res_over_iter_quantiles.txt} ;
					\addplot[black!20] fill between[of=SRK25 and SRK75];
					
					\addplot+[name path=ESRKmin, draw=none] table[y=min_res_esrk]{MATLAB/randn_sparse/lambda=.1//res_over_iter_quantiles.txt} ;
					\addplot+[name path=ESRKmax, draw=none] table[y=max_res_esrk]{MATLAB/randn_sparse/lambda=.1//res_over_iter_quantiles.txt} ;
					\addplot+[green, thick, solid] table[y=median_res_esrk]{MATLAB/randn_sparse/lambda=.1//res_over_iter_quantiles.txt} ;
					\addplot[green!10] fill between[of=ESRKmin and ESRKmax];
					\addplot+[name path=ESRK25, draw=none] table[y=quant25_res_esrk]{MATLAB/randn_sparse/lambda=.1//res_over_iter_quantiles.txt} ;
					\addplot+[name path=ESRK75, draw=none] 	table[y=quant75_res_esrk]{MATLAB/randn_sparse/lambda=.1//res_over_iter_quantiles.txt} ;
					\addplot[green!20] fill between[of=ESRK25 and ESRK75];
					
					\addplot+[name path=SRKREMmin, draw=none] table[y=min_res_hb_double_inexact]{MATLAB/randn_sparse/lambda=.1//res_over_iter_quantiles.txt} ;
					\addplot+[name path=SRKREMmax, draw=none] table[y=max_res_hb_double_inexact]{MATLAB/randn_sparse/lambda=.1//res_over_iter_quantiles.txt} ;
					\addplot+[red, thick] table[y=median_res_hb_double_inexact]{MATLAB/randn_sparse/lambda=.1//res_over_iter_quantiles.txt} ;
					\addplot[red!10, fill opacity = 0.5] fill between[of=SRKREMmin and SRKREMmax];
					\addplot+[name path=SRKREM25, draw=none] table[y=quant25_res_hb_double_inexact]{MATLAB/randn_sparse/lambda=.1//res_over_iter_quantiles.txt} ;
					\addplot+[name path=SRKREM75, draw=none] 			table[y=quant75_res_hb_double_inexact]{MATLAB/randn_sparse/lambda=.1//res_over_iter_quantiles.txt} ;
					\addplot[red!20] fill between[of=SRKREM25 and SRKREM75];		
					
					\addplot+[name path=SRKEMmin, draw=none] table[y=min_res_hb_opt_beta]{MATLAB/randn_sparse/lambda=.1//res_over_iter_quantiles.txt} ;
					\addplot+[name path=SRKEMmax, draw=none] table[y=max_res_hb_opt_beta]{MATLAB/randn_sparse/lambda=.1//res_over_iter_quantiles.txt} ;
					\addplot+[blue, densely dotted] table[y=median_res_hb_opt_beta]{MATLAB/randn_sparse/lambda=.1//res_over_iter_quantiles.txt} ;
					\addplot[blue!10, fill opacity = 0.5] fill between[of=SRKEMmin and SRKEMmax];
					\addplot+[name path=SRKEM25, draw=none] table[y=quant25_res_hb_opt_beta]{MATLAB/randn_sparse/lambda=.1//res_over_iter_quantiles.txt} ;
					\addplot+[name path=SRKEM75, draw=none] 				table[y=quant75_res_hb_opt_beta]{MATLAB/randn_sparse/lambda=.1/res_over_iter_quantiles.txt} ;
					\addplot[blue!20] fill between[of=SRKEM25 and SRKEM75];	
					
				\end{axis}
			\end{tikzpicture}
		
			&
			
			\begin{tikzpicture}
				\begin{axis}[
					thin,smooth,no markers,
					ymode=log,     
					width = 0.45\textwidth, 
					xlabel = {$k$}, 
					ylabel = {$\|\A x_k-b\|_2/\|b\|_2$}, 
					xtick = {500, 1000, 1500}, 
					ytick = {1e-12, 1e-8, 1e-4, 1e0},  
					ymin = 1e-16, ymax = 1e3,
					]
					
					\addplot+[name path=SRKmin, draw=none] table[y=min_res_srk]{MATLAB/randn_sparse/lambda=5/res_over_iter_quantiles.txt} ;
					\addplot+[name path=SRKmax, draw=none] table[y=max_res_srk]{MATLAB/randn_sparse/lambda=5/res_over_iter_quantiles.txt} ;
					\addplot+[black, thick] table[y=median_res_srk]{MATLAB/randn_sparse/lambda=5/res_over_iter_quantiles.txt} ;
					\addplot[black!10] fill between[of=SRKmin and SRKmax];
					\addplot+[name path=SRK25, draw=none] table[y=quant25_res_srk]{MATLAB/randn_sparse/lambda=5/res_over_iter_quantiles.txt} ;
					\addplot+[name path=SRK75, draw=none] 	table[y=quant75_res_srk]{MATLAB/randn_sparse/lambda=5/res_over_iter_quantiles.txt} ;
					\addplot[black!20] fill between[of=SRK25 and SRK75];
					
					\addplot+[name path=ESRKmin, draw=none] table[y=min_res_esrk]{MATLAB/randn_sparse/lambda=5/res_over_iter_quantiles.txt} ;
					\addplot+[name path=ESRKmax, draw=none] table[y=max_res_esrk]{MATLAB/randn_sparse/lambda=5/res_over_iter_quantiles.txt} ;
					\addplot+[green, thick, solid] table[y=median_res_esrk]{MATLAB/randn_sparse/lambda=5/res_over_iter_quantiles.txt} ;
					\addplot[green!10] fill between[of=ESRKmin and ESRKmax];
					\addplot+[name path=ESRK25, draw=none] table[y=quant25_res_esrk]{MATLAB/randn_sparse/lambda=5/res_over_iter_quantiles.txt} ;
					\addplot+[name path=ESRK75, draw=none] 	table[y=quant75_res_esrk]{MATLAB/randn_sparse/lambda=5/res_over_iter_quantiles.txt} ;
					\addplot[green!20] fill between[of=ESRK25 and ESRK75];
					
					\addplot+[name path=SRKREMmin, draw=none] table[y=min_res_hb_double_inexact]{MATLAB/randn_sparse/lambda=5/res_over_iter_quantiles.txt} ;
					\addplot+[name path=SRKREMmax, draw=none] table[y=max_res_hb_double_inexact]{MATLAB/randn_sparse/lambda=5/res_over_iter_quantiles.txt} ;
					\addplot+[red, thick] table[y=median_res_hb_double_inexact]{MATLAB/randn_sparse/lambda=5/res_over_iter_quantiles.txt} ;
					\addplot[red!10, fill opacity = 0.5] fill between[of=SRKREMmin and SRKREMmax];
					\addplot+[name path=SRKREM25, draw=none] table[y=quant25_res_hb_double_inexact]{MATLAB/randn_sparse/lambda=5/res_over_iter_quantiles.txt} ;
					\addplot+[name path=SRKREM75, draw=none] 			table[y=quant75_res_hb_double_inexact]{MATLAB/randn_sparse/lambda=5/res_over_iter_quantiles.txt} ;
					\addplot[red!20] fill between[of=SRKREM25 and SRKREM75];		
					
					\addplot+[name path=SRKEMmin, draw=none] table[y=min_res_hb_opt_beta]{MATLAB/randn_sparse/lambda=5/res_over_iter_quantiles.txt} ;
					\addplot+[name path=SRKEMmax, draw=none] table[y=max_res_hb_opt_beta]{MATLAB/randn_sparse/lambda=5/res_over_iter_quantiles.txt} ;
					\addplot+[blue, densely dotted] table[y=median_res_hb_opt_beta]{MATLAB/randn_sparse/lambda=5/res_over_iter_quantiles.txt} ;
					\addplot[blue!10, fill opacity = 0.5] fill between[of=SRKEMmin and SRKEMmax];
					\addplot+[name path=SRKEM25, draw=none] table[y=quant25_res_hb_opt_beta]{MATLAB/randn_sparse/lambda=5/res_over_iter_quantiles.txt} ;
					\addplot+[name path=SRKEM75, draw=none] 				table[y=quant75_res_hb_opt_beta]{MATLAB/randn_sparse/lambda=5/res_over_iter_quantiles.txt} ;
					\addplot[blue!20] fill between[of=SRKEM25 and SRKEM75];	
					
				\end{axis}
			\end{tikzpicture}
			
			\end{tabular} 
			
			\vspace{0.3cm}
			% legend
			\begin{tikzpicture}[yscale=0.16,xscale=0.012,thick]
				\draw[black] (-230,100) --
				(-180,100)node[right]{\small SRK};
				\draw[green] (-80,100) -- 
				(-30,100)node[right]{\small ESRK};
				\draw[blue] (90,100) --
				(140,100)node[right]{\small SRK-EM};
				\draw[red] (280,100) --
				(330,100)node[right]{\small SRK-REM};
			\end{tikzpicture}
			
		\end{center}
		
		\caption{Experiment (I) with normal distributed matrix $\A$, $(m,n) = (200,500)$, $\hat x$ with $10$ normal distributed nonzero entries, $50$ random repeats. Left: $\lambda=0.1$, right: $\lambda=5$. Thick line shows median over all trials, light area is between min and max, darker area indicates 25th and 75th quantile. In SRK-EM, in case that $\|d_k^*\|_2 > \texttt{eps}$, we performed SRK steps.}
		
		\label{fig:randn_quantiles}
	\end{figure}
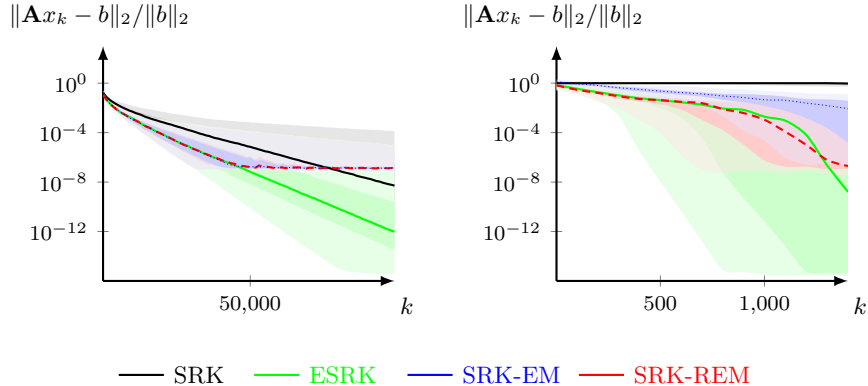

	\begin{table}[htb]
	\begin{center}
		\begin{tabular}{cccc}\toprule
			& CPU time (min) &CPU time (mean) & CPU time (max) \\ 
			\midrule
			SRK & 0.050 & * & * \\
			ESRK & 0.043 & 0.206 & 1.901 \\ 
			SRK-EM & 0.108 & 0.311 & 1.592 \\
			SRK-REM & \bf{0.003} & \bf{0.011} & \bf{0.102} \\
			\bottomrule
			% \hline 
			% SRKHB with $\beta=0.5$ &  & & \\
			% \hline 
			% SRKHB with $\beta=0.96$ & & & \\
			%	\hline 
			% SRKHB with $\beta=0.97$ & & &
		\end{tabular}
	\end{center}
          \caption{CPU time (s) until a relative residual $\|\A x_k-b\|_2 / \|b\|_2 < 10^{-6}$ is achieved in the problem from Experiment (I), cf. Figure~\ref{fig:randn_quantiles}, right subfigure ($\lambda=5$). Here, ``*'' indicates that the relative residual was not achieved after less than $10^5$ iterations. Columns display minimum, average and maximum time over $50$ random repeats. SRK: sparse Kaczmarz, ESRK: exact-step sparse Kaczmarz, SRK-EM: sparse Kaczmarz with exact momentum, SRK-REM: sparse Kaczmarz with relaxed exact momentum.}
		\label{tab:randn_runtime}
	\end{table}

	In experiment (II), we test the methods on matrices from the \emph{SuiteSparseCollection}~\cite{DH11}. To obtain a consistent system, here we chose $\hat x$ with $10\%$ nonzero entries, standard normally distributed, and set $b= \A\hat x$ with the respective matrix $\A$. The results are given in Figure~\ref{fig:SuiteSparse_1} and Figure~\ref{fig:SuiteSparse_2}. We observe different convergence behaviour in several examples. In the examples in Figure~\ref{fig:SuiteSparse_1}, the SRK-REM method converges fastest initially w.r.t. runtime, and is outperformed by the SRK-EM method for small residuals, which then converges faster than all other methods. In the examples in Figure~\ref{fig:SuiteSparse_2}, the SRK method is not visibly accelerated by our methods even over iterations. In the last example, all accelerations except the SRK-REM method perform worse than vanilla SRK.

\begin{figure}[htb]

\begin{center}

\begin{tikzpicture}
	%\pgfplotsset{small}
	\matrix {
		
		\begin{axis}[
			thin,smooth,no markers,
			ymode=log,     
			width = 0.45\textwidth, 
			xlabel = {$k$}, 
			ylabel = {$\|\A x_k-b\|_2/\|b\|_2$}, 
			xtick = {500}, 
			ytick = {1e-12, 1e-8, 1e-4, 1e0},  
			ymin = 1e-12, ymax = 1e0,
			]
			
			\addplot [black] table[x=k, y=res_srk]{MATLAB/SuiteSparseCollection/results/ash85/res_over_iter.txt} ;	
			\addplot [green] table[x=k, y=res_esrk]{MATLAB/SuiteSparseCollection/results/ash85/res_over_iter.txt} ;	
			\addplot [blue] table[x=k, y=res_hb_opt_beta]{MATLAB/SuiteSparseCollection/results/ash85/res_over_iter.txt} ;	
			\addplot [red] table[x=k, y=res_hb_double_inexact]{MATLAB/SuiteSparseCollection/results/ash85/res_over_iter.txt} ;	
			
		\end{axis}
		&
		\begin{axis}[
			thin,smooth,no markers,
			ymode=log,     
			width = 0.45\textwidth, 
			xlabel = {\hspace{1cm} time (s)}, 
			ylabel = {$\|\A x_k-b\|_2/\|b\|_2$}, 
			xtick = {0.02}, 
			ytick = {1e-12, 1e-8, 1e-4, 1e0},  
			ymin = 1e-12, ymax = 1e0,
			]
			
			\addplot [black] table[x=rt_srk, y=res_srk]{MATLAB/SuiteSparseCollection/results/ash85/res_over_iter.txt} ;	
			\addplot [green] table[x=rt_esrk, y=res_esrk]{MATLAB/SuiteSparseCollection/results/ash85/res_over_iter.txt} ;	
			\addplot [blue] table[x=rt_hb_opt_beta, y=res_hb_opt_beta]{MATLAB/SuiteSparseCollection/results/ash85/res_over_iter.txt} ;	
			\addplot [red] table[x=rt_hb_double_inexact, y=res_hb_double_inexact]{MATLAB/SuiteSparseCollection/results/ash85/res_over_iter.txt} ;	
			
		\end{axis} \\
		
	\begin{axis}[
		thin,smooth,no markers,
		ymode=log,     
		width = 0.45\textwidth, 
		xlabel = {$k$}, 
		ylabel = {$\|\A x_k-b\|_2/\|b\|_2$}, 
		xtick = {10000}, 
		ytick = {1e-12, 1e-8, 1e-4, 1e0},  
		ymin = 1e-12, ymax = 1e0,
		]
		
		\addplot [black] table[x=k, y=res_srk]{MATLAB/SuiteSparseCollection/results/well1033/res_over_iter.txt} ;	
		\addplot [green] table[x=k, y=res_esrk]{MATLAB/SuiteSparseCollection/results/well1033/res_over_iter.txt} ;	
		\addplot [blue] table[x=k, y=res_hb_opt_beta]{MATLAB/SuiteSparseCollection/results/well1033/res_over_iter.txt} ;	
		\addplot [red] table[x=k, y=res_hb_double_inexact]{MATLAB/SuiteSparseCollection/results/well1033/res_over_iter.txt} ;	
		
	\end{axis}
	&
	\begin{axis}[
		thin,smooth,no markers,
		ymode=log,     
		width = 0.45\textwidth, 
		xlabel = {\hspace{1cm} time (s)}, 
		ylabel = {$\|\A x_k-b\|_2/\|b\|_2$}, 
		xtick = {0.1}, 
		ytick = {1e-12, 1e-8, 1e-4, 1e0},  
		ymin = 1e-12, ymax = 1e0,
		]
		
		\addplot [black] table[x=rt_srk, y=res_srk]{MATLAB/SuiteSparseCollection/results/well1033/res_over_iter.txt} ;	
		\addplot [green] table[x=rt_esrk, y=res_esrk]{MATLAB/SuiteSparseCollection/results/well1033/res_over_iter.txt} ;	
		\addplot [blue] table[x=rt_hb_opt_beta, y=res_hb_opt_beta]{MATLAB/SuiteSparseCollection/results/well1033/res_over_iter.txt} ;	
		\addplot [red] table[x=rt_hb_double_inexact, y=res_hb_double_inexact]{MATLAB/SuiteSparseCollection/results/well1033/res_over_iter.txt} ;	
		\end{axis}	\\ 	
		};
\end{tikzpicture}

		\vspace{0.3cm}
		% legend
		\begin{tikzpicture}[yscale=0.16,xscale=0.012,thick]
			\draw[black] (-230,100) --
			(-180,100)node[right]{\small SRK};
			\draw[green] (-80,100) -- 
			(-30,100)node[right]{\small ESRK};
			\draw[blue] (90,100) --
			(140,100)node[right]{\small SRK-EM};
			\draw[red] (280,100) --
			(330,100)node[right]{\small SRK-REM};
		\end{tikzpicture}
		
	\end{center}

	\caption{Experiment (II) with matrices from the SuiteSparseCollection, relative residuals over iterations (left) and runtime (right), $\lambda=1$. First row of figures: \texttt{ash85} ($85\times 85$), second row: \texttt{well1033} ($1033\times 320$). Thick line shows median over all trials, light area is between min and max, darker area indicates 25th and 75th quantile. In SRK-EM, in case that $\|d_k^*\|_2 > 10^{-6}$, we executed the step form SRK.}
	
	\label{fig:SuiteSparse_1}

\end{figure}
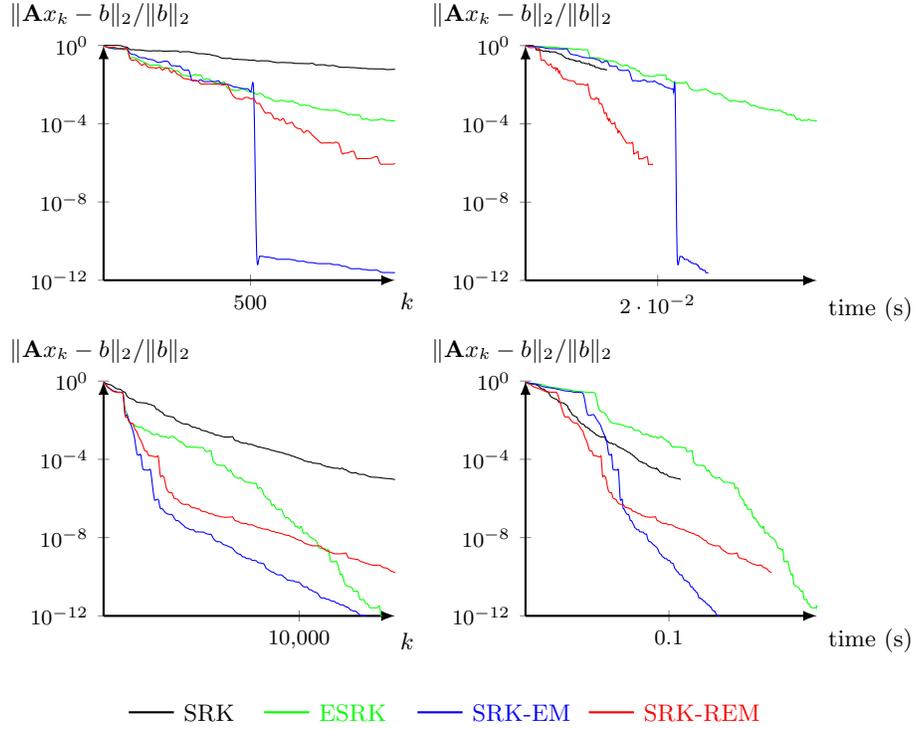

	Finally, we study an academic CT example (experiment (III)) using a fan beam tomography matrix from the \emph{AIR tools} package~\cite{HS12}. We choose $\hat x$ and $b$ as in experiment (II) and set $\lambda=1$. The results are given in Figure~\ref{fig:fanbeamtomo}. We can see that all modifications of the vanilla sparse Kaczmarz method speed up convergence in terms of iterations. However, with respect to computation time the proposed SRK-REM method is the only of our accelerations which actually gives acceleration.
	
	We also comment on an important detail for implementation. If $\|d_k^*\|_2$ becomes small, the condition of the optimization problem in the SRK-EM method can worsen drastically, which then leads to instabilities and high oscillating errors. We therefore replace the if-condition in Algorithm \ref{alg:SMD_EM} by the condition $\|d_k^*\|_2 > \texttt{tol}$ in practice. The choice of \texttt{tol} is critical and depends on the concrete problem instances. For too small values of \texttt{tol}, the method may oscillate, while for too large values of \texttt{tol}, convergence may be slowed down. In experiment (I), we found that $\texttt{tol} = \texttt{eps}$ is a good choice, where $\texttt{eps}\approx 2\cdot 10^{-16}$ is the MATLAB precision value. For $\texttt{tol}=10^{-6}$, convergence was slowed down so much that in $10^5$ iterations a relative residual of $10^{-6}$ was not achieved, which was not the case for any other acceleration in the experiment (Table~\ref{tab:randn_runtime}). In experiment (II) the choice $\texttt{tol}=\texttt{eps}$ lead to high oscillation in the errors and we mostly obtained good results with $\texttt{tol}=10^{-6}$.

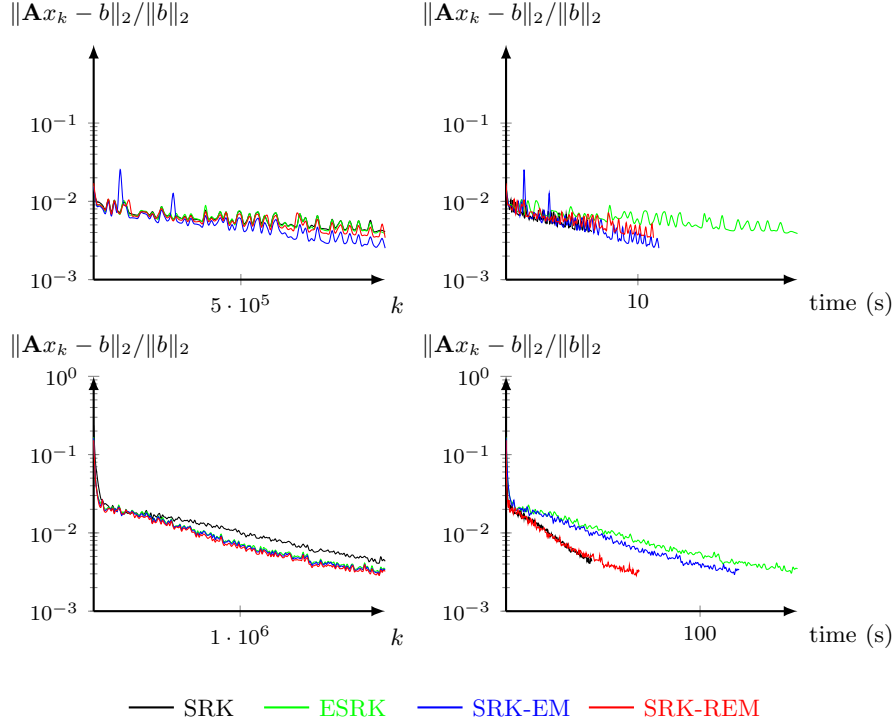
\begin{figure}[htb]
	
	\begin{center}
		
		\begin{tikzpicture}
			%\pgfplotsset{small}
			\matrix {

	\begin{axis}[
	thin,smooth,no markers,
	ymode=log,     
	width = 0.45\textwidth, 
	xlabel = {$k$}, 
	ylabel = {$\|\A x_k-b\|_2/\|b\|_2$}, 
	xtick = {5e5}, 
	ytick = {1e-3, 1e-2, 1e-1},  
	ymin = 1e-3, ymax = 1e0,
	]
	
	\addplot [black] table[x=k, y=res_srk]{MATLAB/SuiteSparseCollection/results/illc1033/res_over_iter.txt} ;	
	\addplot [green] table[x=k, y=res_esrk]{MATLAB/SuiteSparseCollection/results/illc1033/res_over_iter.txt} ;	
	\addplot [blue] table[x=k, y=res_hb_opt_beta]{MATLAB/SuiteSparseCollection/results/illc1033/res_over_iter.txt} ;	
	\addplot [red] table[x=k, y=res_hb_double_inexact]{MATLAB/SuiteSparseCollection/results/illc1033/res_over_iter.txt} ;	
	
\end{axis}
&
\begin{axis}[
	thin,smooth,no markers,
	ymode=log,     
	width = 0.45\textwidth, 
	xlabel = {\hspace{1cm} time (s)}, 
	ylabel = {$\|\A x_k-b\|_2/\|b\|_2$}, 
	xtick = {10}, 
	ytick = {1e-3, 1e-2, 1e-1},  
	ymin = 1e-3, ymax = 1e0,
	]
	
	\addplot [black] table[x=rt_srk, y=res_srk]{MATLAB/SuiteSparseCollection/results/illc1033/res_over_iter.txt} ;	
	\addplot [green] table[x=rt_esrk, y=res_esrk]{MATLAB/SuiteSparseCollection/results/illc1033/res_over_iter.txt} ;	
	\addplot [blue] table[x=rt_hb_opt_beta, y=res_hb_opt_beta]{MATLAB/SuiteSparseCollection/results/illc1033/res_over_iter.txt} ;	
	\addplot [red] table[x=rt_hb_double_inexact, y=res_hb_double_inexact]{MATLAB/SuiteSparseCollection/results/illc1033/res_over_iter.txt} ;	
	
\end{axis} \\ 	

\begin{axis}[
	thin,smooth,no markers,
	ymode=log,     
	width = 0.45\textwidth, 
	xlabel = {$k$}, 
	ylabel = {$\|\A x_k-b\|_2/\|b\|_2$}, 
	xtick = {1e6}, 
	ytick = {1e-3, 1e-2, 1e-1, 1e0},  
	ymin = 1e-3, ymax = 1e0,
	]
	
	\addplot [black] table[x=k, y=res_srk]{MATLAB/SuiteSparseCollection/results/landmark/res_over_iter.txt} ;	
	\addplot [green] table[x=k, y=res_esrk]{MATLAB/SuiteSparseCollection/results/landmark/res_over_iter.txt} ;	
	\addplot [blue] table[x=k, y=res_hb_opt_beta]{MATLAB/SuiteSparseCollection/results/landmark/res_over_iter.txt} ;	
	\addplot [red] table[x=k, y=res_hb_double_inexact]{MATLAB/SuiteSparseCollection/results/landmark/res_over_iter.txt} ;	
	
\end{axis}
&
\begin{axis}[
	thin,smooth,no markers,
	ymode=log,     
	width = 0.45\textwidth, 
	xlabel = {\hspace{1cm} time (s)}, 
	ylabel = {$\|\A x_k-b\|_2/\|b\|_2$}, 
	xtick = {100}, 
	ytick = {1e-3, 1e-2, 1e-1, 1e0},  
	ymin = 1e-3, ymax = 1e0,
	]
	
	\addplot [black] table[x=rt_srk, y=res_srk]{MATLAB/SuiteSparseCollection/results/landmark/res_over_iter.txt} ;	
	\addplot [green] table[x=rt_esrk, y=res_esrk]{MATLAB/SuiteSparseCollection/results/landmark/res_over_iter.txt} ;	
	\addplot [blue] table[x=rt_hb_opt_beta, y=res_hb_opt_beta]{MATLAB/SuiteSparseCollection/results/landmark/res_over_iter.txt} ;	
	\addplot [red] table[x=rt_hb_double_inexact, y=res_hb_double_inexact]{MATLAB/SuiteSparseCollection/results/landmark/res_over_iter.txt} ;	
	
\end{axis} \\ 	
};

\end{tikzpicture}

\vspace{0.3cm}
% legend
\begin{tikzpicture}[yscale=0.16,xscale=0.012,thick]
\draw[black] (-230,100) --
(-180,100)node[right]{\small SRK};
\draw[green] (-80,100) -- 
(-30,100)node[right]{\small ESRK};
\draw[blue] (90,100) --
(140,100)node[right]{\small SRK-EM};
\draw[red] (280,100) --
(330,100)node[right]{\small SRK-REM};
\end{tikzpicture}

\end{center}

\caption{Experiment (II) with matrices from the SuiteSparseCollection, relative residuals over iterations (left) and runtime (right), $\lambda=1$. First row of figures: \texttt{ill1033} ($1033\times 320$), second row: \texttt{landmark} ($71952\times 2704$). Thick line shows median over all trials, light area is between min and max, darker area indicates 25th and 75th quantile. In SRK-EM, in case that $\|d_k^*\|_2 > 10^{-6}$, we executed the step form SRK.}

\label{fig:SuiteSparse_2}

\end{figure}

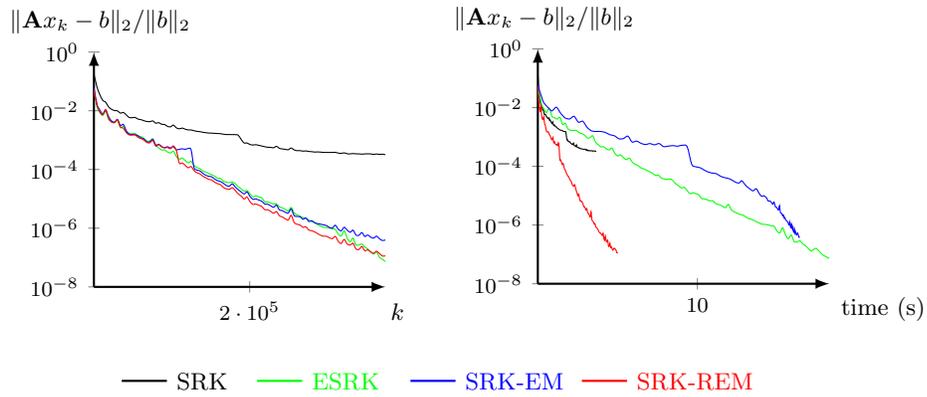
\begin{figure}[htb]

\begin{center}
    
    \begin{tabular}{rl}
        
        \begin{tikzpicture}
            \begin{axis}[
                thin,smooth,no markers,
                ymode=log,     
                width = 0.45\textwidth, 
                xlabel = {$k$}, 
                ylabel = {$\|\A x_k-b\|_2/\|b\|_2$}, 
                xtick = {2e5}, 
                ytick = {1e-8, 1e-6, 1e-4, 1e-2, 1e0},  
                ymin = 1e-8, ymax = 1e0,
                ]
                
                \addplot [black] table[ y=res_srk]{MATLAB/CT/results/fanbeamtomo/res_over_iter.txt} ;	
                \addplot [green] table[ y=res_esrk]{MATLAB/CT/results/fanbeamtomo/res_over_iter.txt} ;	
                \addplot [blue] table[ y=res_hb_opt_beta]{MATLAB/CT/results/fanbeamtomo/res_over_iter.txt} ;	
                \addplot [red] table[ y=res_hb_double_inexact]{MATLAB/CT/results/fanbeamtomo/res_over_iter.txt} ;	
                
            \end{axis} 
        \end{tikzpicture}
        
        &
        
        \begin{tikzpicture}
            \begin{axis}[
                thin,smooth,no markers,
                ymode=log,     
                width = 0.45\textwidth, 
                xlabel = {\hspace{1cm} time (s)}, 
                ylabel = {$\|\A x_k-b\|_2/\|b\|_2$}, 
                xtick = {10}, 
                ytick = {1e-8, 1e-6, 1e-4, 1e-2, 1e0},  
                ymin = 1e-8, ymax = 1e0,
                ]
                
                \addplot [black] table[x=rt_srk, y=res_srk]{MATLAB/CT/results/fanbeamtomo/res_over_iter.txt} ;	
                \addplot [green] table[x=rt_esrk, y=res_esrk]{MATLAB/CT/results/fanbeamtomo/res_over_iter.txt} ;	
                \addplot [blue] table[x=rt_hb_opt_beta, y=res_hb_opt_beta]{MATLAB/CT/results/fanbeamtomo/res_over_iter.txt} ;	
                \addplot [red] table[x=rt_hb_double_inexact, y=res_hb_double_inexact]{MATLAB/CT/results/fanbeamtomo/res_over_iter.txt} ;	
                
            \end{axis} 
        \end{tikzpicture}
        
    \end{tabular} 
    
    \vspace{0.3cm}
    % legend
    \begin{tikzpicture}[yscale=0.16,xscale=0.012,thick]
        \draw[black] (-230,100) --
        (-180,100)node[right]{\small SRK};
        \draw[green] (-80,100) -- 
        (-30,100)node[right]{\small ESRK};
        \draw[blue] (90,100) --
        (140,100)node[right]{\small SRK-EM};
        \draw[red] (280,100) --
        (330,100)node[right]{\small SRK-REM};
    \end{tikzpicture}
    
\end{center}

\caption{Experiment (III) with an academic CT matrix $\A$ arising from fan beam tomography, $(m,n) = (3780,900)$, $\hat x$ with $10\%$ normal distributed nonzero entries, $\lambda=1$. In SRK-EM, in case that $\|d_k^*\|_2 > 10^{-6}$, we executed the step form SRK.}

\label{fig:fanbeamtomo}
\end{figure}

\clearpage

\bibliography{refs}

\bibliographystyle{abbrv}
	
\end{document}